\documentclass[sts,
	reqno,
]{imsart}

\usepackage{amsthm, amssymb, amsmath}
\usepackage{empheq}
\usepackage{euscript}
\usepackage{mathtools}
\usepackage{mdframed}
\usepackage{tcolorbox}
\usepackage{color}
\usepackage{sectsty}
\usepackage{tikz}
\usetikzlibrary{arrows}
\usepackage[utf8]{inputenc}
\usepackage[shortlabels]{enumitem}
\usepackage{url}  
\usepackage{cases} 
\definecolor{labelkey}{rgb}{0.0, 0.8, 0.3}
\usepackage[top=2.5cm,bottom=2.5cm,left=2.5cm,right=2.5cm,includeheadfoot,asymmetric]{geometry} 

\usepackage{algorithm}
\usepackage{algpseudocode}
\makeatletter
\let\c@author\relax
\makeatother
\usepackage[giveninits=true, style=alphabetic]{biblatex}
\usepackage{caption}
\usepackage{microtype}
\usepackage{minitoc}

\usepackage{ragged2e}
\usepackage{subcaption}
\usepackage{wrapfig}

\usepackage{custom}
\usepackage{hyperref}
\hypersetup{
  colorlinks = true,
  urlcolor = blue, 
  linkcolor = blue,
  citecolor = blue}

\usepackage{soul}

\numberwithin{equation}{section}
\allsectionsfont{\sffamily}

\declaretheorem[name=Lemma]{lemma}
\declaretheorem[name=Proposition]{prop}
\declaretheorem[name=Remark, style=remark]{rema}
\declaretheorem[name=Theorem]{thm}
\DeclareMathOperator{\clip}{clip}
\DeclareMathOperator*{\minimize}{minimize}
\DeclareMathOperator{\GM}{\textsc{GM}}
\renewcommand{\abstractname}{Abstract}

\noptcrule{}

\begin{document} 

\begin{frontmatter}

	\title{Averaging on the Bures--Wasserstein manifold: dimension-free convergence of gradient descent}
	\runtitle{Averaging on the Bures--Wasserstein manifold}
 
	\author{Jason M.\ Altschuler \hfill jasonalt@mit.edu \\
	Sinho Chewi \hfill schewi@mit.edu \\
	Patrik Gerber \hfill prgerber@mit.edu \\
	Austin J.\ Stromme \hfill astromme@mit.edu}

	\address{{Jason M.\ Altschuler}\\
		{Department of EECS} \\
		{Massachusetts Institute of Technology}\\
		{77 Massachusetts Avenue,}\\
		{Cambridge, MA 02139-4307, USA}\\
	}
	
		\address{{Sinho Chewi}\\
		{Department of Mathematics} \\
		{Massachusetts Institute of Technology}\\
		{77 Massachusetts Avenue,}\\
		{Cambridge, MA 02139-4307, USA}\\
	}
	
			\address{{Patrik Gerber}\\
		{Department of Mathematics} \\
		{Massachusetts Institute of Technology}\\
		{77 Massachusetts Avenue,}\\
		{Cambridge, MA 02139-4307, USA}\\
	}
				\address{{Austin J.\ Stromme}\\
		{Department of EECS} \\
		{Massachusetts Institute of Technology}\\
		{77 Massachusetts Avenue,}\\
		{Cambridge, MA 02139-4307, USA}\\
}

\runauthor{Altschuler, Chewi, Gerber, and Stromme}

\begin{abstract}
    We study first-order optimization algorithms for computing the barycenter of Gaussian distributions with respect to the optimal transport metric.
    Although the objective is geodesically non-convex,
    Riemannian GD empirically converges rapidly, in fact faster than off-the-shelf methods such as Euclidean GD and SDP solvers.
    This stands in stark contrast to the best-known theoretical results for Riemannian GD, which depend
    exponentially on the dimension. 
    In this work, we prove new geodesic convexity results which provide stronger control of the iterates, yielding
    a dimension-free convergence rate.
    Our techniques also enable the analysis of two related notions of averaging, the entropically-regularized barycenter and the geometric median, providing the first convergence guarantees for Riemannian GD for these problems.
\end{abstract}

\end{frontmatter}

\doparttoc
\faketableofcontents

\section{Introduction}

Averaging multiple data sources is among the most classical and fundamental subroutines in data science. However, a modern challenge is that data is often more complicated than points in $\R^d$. In this paper, we study the task of averaging probability distributions on $\R^d$, a setting that commonly arises in machine learning and statistics~\cite{cuturi2014barycenters,ho2017multilevel,srivastava2018scalable,dognin2019wasserstein}, computer vision and graphics~\cite{rabin2011wasserstein,solomon2015convolutional}, probability theory~\cite{knottsmith1994cyclicmonotonicity,ruschendorf2002n}, and signal processing~\cite{elvander2020multi}; see also the surveys~\cite{peyre2019computational, panaretos2019statistical} and the references within.

The Wasserstein barycenter~\cite{agueh2011barycenter} has emerged as a particularly canonical notion of average. Formally, let $\mc P_2(\R^d)$ denote the space of probability measures on $\R^d$ with finite second moment, let $P$ be a probability measure over $\mc P_2(\R^d)$, and let $W_2$ denote the $2$-Wasserstein distance (i.e., the standard optimal transport distance). Then, the \emph{Wasserstein barycenter} of $P$ is a solution of
\begin{align}\label{eq:w2_bary}
    \minimize_{b \in \mc P_2(\R^d)} \qquad \int W_2^2(b, \cdot) \, \D P\,.
\end{align}
A related notion of average is the \emph{entropically-regularized Wasserstein barycenter} of $P$~\cite{kroshnin2018barycentersmk, bigotcazellespapadakis2019penbarycenter, carliereichingerkroshnin2020entropicbarycenter}, which is defined to be a solution of
\begin{align}\label{eq:w2_reg_bary}
    \minimize_{b \in \mc P_2(\R^d)} \qquad \int W_2^2(b,\cdot) \, \D P + \on{ent}(b)\,,
\end{align}
where $\on{ent}$ is an entropic penalty which allows for incorporating prior knowledge into the average.
Lastly, a third related notion of average with better robustness properties (e.g., with a breakdown point of 50\%~\cite{FKJ}) is the \emph{Wasserstein geometric median} of $P$, which is defined to be a solution of 
\begin{align}\label{eq:w2_median}
    \minimize_{b \in \mc P_2(\R^d)} \qquad\int W_2(b, \cdot) \, \D P\,.
\end{align}

Importantly, while these three notions of average can be defined using other metrics in lieu of $W_2$, the Wasserstein distance is critical for many applications since it enables capturing geometric features of the distributions~\cite{cuturi2014barycenters}.

The many applications of Wasserstein barycenters and geometric medians (see, e.g.,~\cite{carlier2010matching, rabin2011wasserstein, cuturi2014barycenters, gramfort2015fast, rabin_papadakis_2015,
solomon2015convolutional, bonneel2016wasserstein, srivastava2018scalable, legouic2020fairness}) have inspired significant research into their mathematical and statistical properties since their introduction roughly a decade ago~\cite{agueh2011barycenter,rabin2011wasserstein}. 
For instance, on the mathematical side it is known that under mild conditions, the barycenter and geometric median exist, are unique, and admit dual formulations related to multimarginal optimal transport problems~\cite{carlier2010matching,agueh2011barycenter,carlier2015numerical}.
On the statistical side,~\cite{panaretos2016point, agueh2017centrale, legouic2015consistency, bigot2018barycenter, kroshnin2019barycenters, ahidarcoutrix2018convergence, legouic2022fast} provide finite-sample and asymptotic statistical guarantees for estimating the Wasserstein barycenter from samples.

However, computing these objects is challenging because of two fundamental obstacles. The first is that in general, barycenters and geometric medians can be complicated distributions which are much harder to represent (even approximately) than the input distributions.
The second is that generically, these problems are computationally hard in high dimensions. For instance, Wasserstein barycenters and geometric medians of discrete distributions are NP-hard to compute (even approximately) in high dimension~\cite{altschulerboixadsera2022barycenterhard}.

\paragraph*{Algorithms for averaging on the Bures--Wasserstein manifold.} Nevertheless, these computational obstacles can be potentially averted in parametric settings. This paper as well as most of the literature~\cite{esteban2016barycenters, zemel2019procrustes, chewietal2020buresgd, backhoffveraguas2022barycenters} on parametric settings focuses on the commonly arising setting where $P$ is supported on Gaussian distributions.\footnote{In the setting of Gaussian distributions, the Wasserstein barycenter was first studied in the 1990s~\cite{olkin1993maximum,knottsmith1994cyclicmonotonicity}.} As noted in~\cite{esteban2016barycenters}, the Gaussian case also encompasses general location-scatter families.

There are two natural families of approaches for designing averaging algorithms in this setting. Both exploit the fact that modulo a simple re-centering of all distributions, the relevant space of probability distributions is isometric
to the \emph{Bures--Wasserstein manifold}, i.e., the cone of positive semidefinite matrices equipped with the Bures--Wasserstein metric (background is given in Section~\ref{scn:preliminaries}). 

The first approach is simply to recognize the (regularized) Wasserstein barycenter problem as a convex optimization problem over the space of positive semidefinite matrices and apply off-the-shelf methods such as Euclidean gradient descent or semidefinite programming solvers. However, these methods have received little prior attention for good reason: they suffer from severe scalability and parameter-tuning issues (see Section~\ref{scn:bures_gd_numerics} for numerics). Briefly, the underlying issue is that these algorithms operate in the standard Euclidean geometry rather than the natural geometry of optimal transport. Moreover, this approach does not apply to the Wasserstein geometric median problem because even in one dimension, it is non-convex in the Euclidean geometry.

A much more effective approach in practice (see Section~\ref{scn:bures_gd_numerics} for numerics) is to exploit the geometry of the Bures--Wasserstein manifold via geodesic optimization. 
Prior work has extensively pursued this direction, investigating the effectiveness of (stochastic) Riemannian gradient descent for computing Wasserstein barycenters, see, e.g., \cite{esteban2016barycenters, zemel2019procrustes, chewietal2020buresgd, backhoffveraguas2022barycenters}. 

\paragraph*{Challenges for geodesic optimization over the Bures--Wasserstein manifold.} Although geodesic optimization is natural for this problem, it comes with several important obstacles: the non-negative curvature of the Bures--Wasserstein manifold necessitates new tools for analysis, and moreover both the barycenter and geometric median problems are \emph{non-convex} in the Bures--Wasserstein geometry. (These two issues are in fact intimately related, see Appendix~\ref{scn:curv_bary}.) This prevents applying standard results in the geodesic optimization literature (see, e.g.,~\cite{zhangsra2016geodesicallycvx,boumal2020introduction}) since in general it is only possible to prove local convergence guarantees for non-convex problems. 

\par For the Wasserstein barycenter problem, it is possible to interpret Riemannian gradient descent (with step size one) as a fixed-point iteration, and through this lens establish asymptotic convergence~\cite{esteban2016barycenters, zemel2019procrustes, backhoffveraguas2022barycenters}.
Obtaining non-asymptotic rates of convergence is more challenging because it requires developing quantitative proxies for the standard convexity inequalities needed to analyze gradient descent. The first such result was achieved by~\cite{chewietal2020buresgd}, showing that Riemannian gradient descent converges to the Wasserstein barycenter at a linear rate. Yet their convergence rate depends exponentially on the dimension $d$, and also their work does not extend to the Wasserstein geometric median or regularized Wasserstein barycenter.

\subsection{Contributions}

In this paper, we analyze first-order optimization algorithms on the Bures--Wasserstein manifold. We summarize our main results here and overview our techniques in the next section.

\paragraph*{From exponential dimension dependence to dimension-free rates.} In Section~\ref{scn:bary}, we show that for the Wasserstein barycenter problem, Riemannian gradient descent enjoys dimension-free convergence rates (Theorem~\ref{thm:bures_gd}). This eliminates the exponential dimension dependence of state-of-the-art convergence rates~\cite{chewietal2020buresgd}, which aligns with the empirical performance of this algorithm (see Figure~\ref{fig:gd and sgd vs d}). It also stands in sharp contrast to the setting of discrete distributions in which there are computational complexity barriers to achieving even polynomial dimension dependence~\cite{altschulerboixadsera2022barycenterhard}.

Moreover, in Theorem~\ref{thm:bures_avg_case}, we give a refinement of this result which replaces the worst-case assumption of uniform bounds on the matrices' eigenvalues with significantly weaker average-case assumptions.

\begin{figure}
\centering
\includegraphics[width=0.5\textwidth]{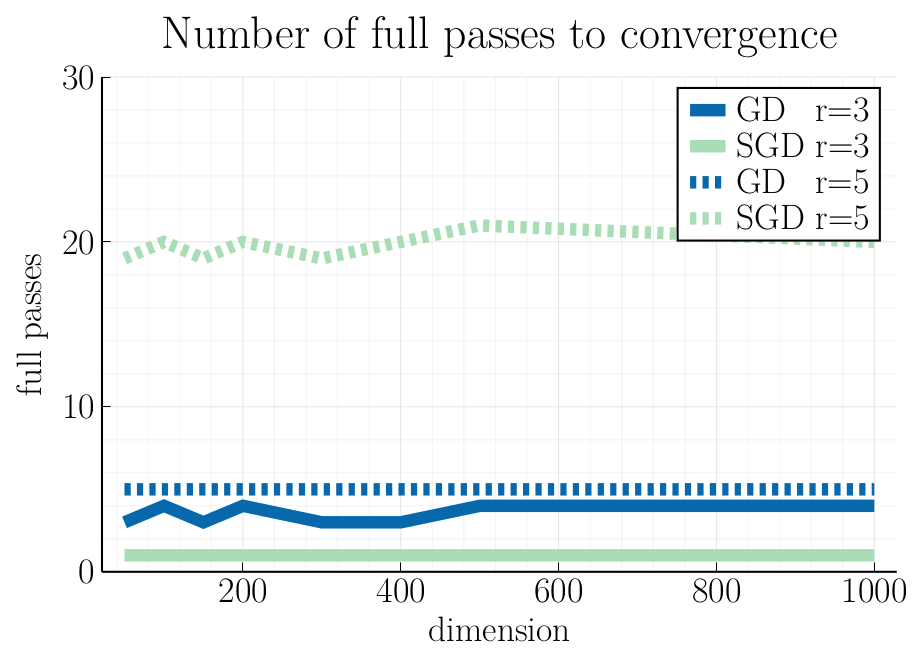}
        \caption{Passes until convergence error $10^{-r}$ to the barycenter, for $r \in \{3, 5\}$. 
        This is \emph{dimension independent} for Riemannian GD and SGD---consistent with our main results. Details in Section~\ref{scn:bary}.}    
        \label{fig:gd and sgd vs d}
\end{figure}

\paragraph*{Beyond barycenters.} In Sections~\ref{scn:reg_bary} and~\ref{scn:median}, we show how our analysis techniques also enable proving fast convergence of Riemannian gradient descent for computing regularized barycenters (Theorem~\ref{thm:ent_main}) and geometric medians (Theorem~\ref{thm:median_guarantee}). To the best of our knowledge, these are the first guarantees for Riemannian gradient descent for notions of averaging on the Bures--Wasserstein manifold beyond the barycenter.

\subsection{Techniques}

Here we briefly sketch the specific technical
challenges we face and how we address them to analyze Riemannian gradient descent for the three notions of
Bures--Wasserstein average: barycenter, regularized barycenter, and median.
Although each analysis necessarily exploits particularities of its own
objective,
the common structure between these analyses may be of
interest for studying other geodesically non-convex optimization problems.

\paragraph*{Overcoming non-convexity.}
As we discuss in Appendix~\ref{scn:curv_bary}, there is a close connection between the second-order
behavior of these objective functionals and the non-negative curvature of the Bures--Wasserstein manifold.
In particular, while non-negative curvature is used to prove smoothness properties for the three functionals, it also leads to them all being geodesically non-convex.
To circumvent this issue, we
establish gradient domination conditions,
also known as Polyak--\L{}ojasiewicz inequalities~\cite{kariminutinischmidt2016pl}, which intuitively are quantitative proxies for strong convexity in the non-convex setting. Proving such inequalities
requires synthesizing general optimization principles with specialized arguments based on the theory of optimal transport. We ultimately show that these inequalities hold with
constants depending on the conditioning of the iterates, i.e., the ratio between the maximum and minimum eigenvalues of the corresponding covariance matrices.

\paragraph*{Overcoming ill-conditioned iterates.} 
So long as smoothness and gradient domination inequalities hold
at the current iterate, standard optimization results guarantee that the next iterate of gradient descent makes progress.
However,
the amount of progress
degrades if the iterates are poorly conditioned.
Thus the second major obstacle is to control the regularity of the iterates. Here, the primary technical tool is shared across
the analyses. Informally, it states that if the objective is a sum of functions,
each of whose gradients point towards well-conditioned matrices, then the gradient descent iterates remain well-conditioned. Formally, this is captured by the following geometric result, which may be of independent interest. Below, $\psd$ denotes the set of $d \times d$ positive definite matrices. See Appendix~\ref{scn:gen_geod_cvxty} for a review of the relevant geometric concepts, and see Appendix~\ref{scn:proof_for_geod_cvx} for the proof, discussion of tightness, and complementary results.

\begin{thm}\label{thm:sqrt_lambda_min_concave}
    Let $0 < \alpha \le \beta < \infty$. Let $Q$ be any distribution over $\psd$ such that
    \begin{align*}
        \Bigl(\int \sqrt{\lambda_{\min}(\Sigma)} \, \D Q(\Sigma)\Bigr)^2 \ge \alpha \qquad\text{and}\qquad\int \lambda_{\max}(\Sigma) \, \D Q(\Sigma) \le \beta\,.
    \end{align*}
    Then, for any matrix $\Sigma_0$ with eigenvalues bounded below by $\frac{\alpha}{4}$ and any $0 \le \eta \le \frac{\alpha}{2\beta}$, the generalized barycenter of $(1-\eta) \, \delta_{\Sigma_0} + \eta \, Q$ at $\Sigma_0$  also has eigenvalues lower bounded by $\frac{\alpha}{4}$.
\end{thm}

Using this theorem in conjunction with careful analysis of the objective functions, we establish global convergence guarantees for first-order geodesic optimization.

In an earlier version of our paper, we incorrectly claimed that $-\sqrt{\lambda_{\min}}$ and $\sqrt{\lambda_{\max}}$ are convex along generalized geodesics, which is stronger than Theorem~\ref{thm:sqrt_lambda_min_concave}. (The error in the proof stemmed from our use of an incorrect result in the literature, namely~\cite[Corollary 3.5]{lawsonlim2001geometricmean}.) This version of the paper fixes this issue; see Remark~\ref{rem:what_went_wrong} for a detailed discussion.

\subsection{Other related work}\label{scn:prior_work}

\paragraph*{Averaging on curved spaces.} Averages such as barycenters and medians on curved spaces 
have become popular due to the applications in machine learning, computer vision, analysis, radar signal processing~\cite{ABY}, and brain-computer interfaces~\cite{yger16,congedo17}. While their mathematical properties such as existence and uniqueness are fairly well-understood~\cite{afsari2011riemannian}, their computation is an active area of research~\cite{Weiszfeld,VardiZhang,sturm2003npc,yang2010riemannian,bini2013computing, bacak,OhtaPafia}. 
For the Wasserstein barycenter problem in particular, there have been a multitude of approaches proposed for both the discrete setting (see, e.g.,~\cite{cuturi2014barycenters,benamou2015iterative,carlier2015numerical,borgwardt2017strongly,pmlr-v97-kroshnin19a,Dvi20,linetal2020barycenter,haasler2020multi,guminov2021accelerated,altschuler2021wasserstein, lin2022complexity}) and the continuous setting (see, e.g.,~\cite{cohen2020estimating,fan2020scalable,lietal2020regbarycenters, korotinetal2021barycenter}).

\paragraph*{Optimal transport and regularization.}
Our work on the entropically-regularized barycenter follows a long and fruitful interplay between optimal transport and entropic regularization.
Entropic regularization yields computational speedups~\cite{wilson1969use,cuturi2013sinkhorn, altschulernilesweedrigollet2017sinkhorn, peyre2019computational}, brings fundamental connections to statistical problems such as the Schr\"odinger bridge~\cite{schrodinger1931umkehrung,leonard2013survey} and maximum likelihood estimation~\cite{rigolletweed2018entropicotmle}, and enjoys much more regularity from a PDE perspective~\cite{leonard2012schrodingermk, carliereichingerkroshnin2020entropicbarycenter}, which in turn has been useful for rigorously establishing functional inequalities~\cite{ledoux2018remarks, fathigozlanprodhomme2020caffarelli, gentiletal2020entropichwi}.
Entropic regularization of optimal transport specifically between Gaussians has also been extensively studied in the literature~\cite{del2020statistical,janati2020entropic,mallasto2020entropy}.

\section{Preliminaries}\label{scn:preliminaries}

We write $\sym$ for the space of symmetric $d\times d$ matrices, and $\psd$ for the open subset of $\sym$ consisting of positive definite matrices.

Given probability measures $\mu$ and $\nu$ on $\R^d$ with finite second moment, the $2$-Wasserstein distance between $\mu$ and $\nu$ is defined as
\begin{align}\label{eq:defn_w2}
    W_2^2(\mu,\nu)
    & \deq  \inf_{\pi \in \Pi(\mu,\nu)} \int \norm{x-y}^2 \, \D \pi(x,y)\,,
\end{align}
where $\Pi(\mu,\nu)$ denotes the set of couplings of $\mu$ and $\nu$, i.e., the set of probability measures on $\R^d\times\R^d$ whose marginals are respectively $\mu$ and $\nu$. If $\mu$ and $\nu$ admit densities with respect to the Lebesgue measure on $\R^d$, then the infimum is attained, and the optimal coupling is supported on the graph of a map, i.e., there exists a map $T : \R^d\to\R^d$ such that for $\pi$-a.e.\ $(x,y)\in\R^d\times \R^d$, it holds that $y = T(x)$. The map $T$ is called the \emph{optimal transport map} from $\mu$ to $\nu$.

We refer readers to~\cite{villani2003topics, santambrogio2015ot} for an introduction to optimal transport, and to~\cite{docarmo1992riemannian} and~\cite[Appendix A.1]{chewietal2020buresgd} for background on Riemannian geometry.
The Riemannian structure of optimal transport was introduced in the seminal work~\cite{otto2001porousmedium}; detailed treatments can be found in~\cite{ambrosio2008gradient, villani2009ot}, see also~\cite[Appendix A.2]{chewietal2020buresgd} and Appendix~\ref{scn:bw_background} for a quick overview.

In this paper, we mainly work with centered Gaussians, which can be identified with their covariance matrices. (Extensions to the non-centered case are also discussed in the next sections.) We abuse notation via this identification: given $\Sigma,\Sigma' \in \psd$, we write $W_2(\Sigma,\Sigma')$ for the $2$-Wasserstein distance between centered Gaussians with covariance matrices $\Sigma$, $\Sigma'$ respectively. Throughout, all Gaussians of interest are non-degenerate; that is, their covariances are non-singular.

The Wasserstein distance has a closed-form expression for Gaussians:
\begin{align}\label{eq:w2_formula}
    W_2^2(\Sigma,\Sigma')
    &= \tr\bigl[\Sigma + \Sigma' - 2 \, {(\Sigma^{1/2} \Sigma' \Sigma^{1/2})}^{1/2}\bigr]\,.
\end{align}
Also, the optimal transport map from $\Sigma$ to $\Sigma'$ is the symmetric matrix
\begin{align}\label{eq:gaussian_ot_map}
    T_{\Sigma\to\Sigma'}
    &= \Sigma^{-1/2} \, {(\Sigma^{1/2} \Sigma' \Sigma^{1/2})}^{1/2} \, \Sigma^{-1/2}
    = \GM(\Sigma^{-1}, \Sigma').
\end{align}
Above, $\GM(A,B)  \deq  A^{1/2} \, {(A^{-1/2} B A^{-1/2})}^{1/2} A^{1/2}$ denotes the matrix geometric mean between two positive semidefinite matrices~\cite[Ch.\ 4]{bhatia2009positive}.
The Wasserstein distance on $\psd$ in fact arises from a Riemannian metric, which was first introduced by Bures in~\cite{bures1969}. Hence, the Riemannian manifold $\psd$ endowed with this Wasserstein distance is referred to as the \emph{Bures--Wasserstein space}. The geometry of this space is studied in detail in~\cite{modin2017matrixdecomposition, bhatiajainlim2019bures}. For completeness, we provide additional background on the Bures--Wasserstein manifold in Appendix~\ref{scn:bw_background}.

\section{Barycenters}\label{scn:bary}

In this section, we consider the Bures--Wasserstein barycenter
\begin{align*}
    \Sigstar \in \argmin_{\Sigma \in \psd} \int W_2^2(\Sigma, \cdot) \, \D P\,.
\end{align*}
We refer to the introduction for a discussion of the past work on the Bures--Wasserstein barycenter. We also remark that the case when $P$ is supported on possibly non-centered Gaussians is easily reduced to the centered case; see the discussion in~\cite[\S 4]{chewietal2020buresgd}.

\subsection{Algorithms}

We consider both Riemannian gradient descent (GD) and Riemannian stochastic gradient descent (SGD) algorithms for computing the Bures--Wasserstein barycenter, which are given as Algorithm~\ref{ALG:GD} and Algorithm~\ref{ALG:SGD} respectively. GD is useful for computing high-precision solutions due to its linear convergence (Theorem~\ref{thm:bures_gd}), and SGD is useful for large-scale or online settings because of its cheaper updates. We refer to~\cite{zemel2019procrustes, chewietal2020buresgd} for the derivation of the updates. Here, $\Sigma_0$ is the initialization, which can be taken to be any matrix in the support of $P$. For SGD, we also require a sequence ${(\eta_t)}_{t=1}^T$ of step sizes and a sequence ${(K_t)}_{t=1}^T$ of i.i.d.\ samples from $P$. 

\noindent
\begin{minipage}{0.45\textwidth}
\centering
\begin{algorithm}[H]
  \caption{GD for Barycenters}\label{ALG:GD}
  \begin{algorithmic}[1]
    \Procedure{Bary-GD}{$\Sigma_0, \eta, P, T$}
      \For{$t = 1, \ldots, T$}
        \State $
        S_t \gets (1-\eta)\,I_d + \eta \int \GM(\Sigma_{t-1}^{-1},\Sigma) \, \D P(\Sigma)
        $
        \State $ \Sigma_{t} \gets S_t\Sigma_{t - 1} S_t$
      \EndFor
      \State \Return $\Sigma_{T}$
    \EndProcedure
  \end{algorithmic}
\end{algorithm}
\end{minipage}
\hfill
\begin{minipage}{0.5\textwidth}
\centering
\begin{algorithm}[H]
  \caption{SGD for Barycenters}\label{ALG:SGD}
  \begin{algorithmic}[1]
    \Procedure{Bary-SGD}{$\Sigma_0, {(\eta_t)}_{t = 1}^T, {(K_t)}_{t = 1}^T$}
      \For{$t = 1, \ldots, T$}
        \State $
        \hat{S}_t \gets (1-\eta_t)\, I_d + \eta_t \GM(\Sigma_{t-1}^{-1}, K_t)
        $
        \State $ \Sigma_{t} \gets \hat{S}_t \Sigma_{t - 1} \hat{S}_t$
      \EndFor
      \State \Return $\Sigma_{T}$
    \EndProcedure
  \end{algorithmic}
\end{algorithm}
\end{minipage}

\medskip{}

Note that whereas SGD requires choosing step sizes, for GD we can simply use step size $1$ in practice, as justified in~\cite{zemel2019procrustes}. However, for our theoretical results, we will require choosing a step size $\eta < 1$ for GD as well.

\subsection{Convergence guarantees}

Denote the barycenter functional by $F(\Sigma)  \deq  \frac{1}{2} \int W_2^2(\Sigma,\cdot) \, \D P$, and denote the \emph{variance} of $P$ by $\var P  \deq  2F(\Sigstar)$. We assume that $P$ is supported on matrices whose eigenvalues lie in the range $[\lambda_{\min},\lambda_{\max}]$, and we let $\kappa  \deq  \lambda_{\max}/\lambda_{\min}$ denote the condition number. Whereas the previous state-of-the-art convergence analysis for Algorithms~\ref{ALG:GD} and~\ref{ALG:SGD} in~\cite{chewietal2020buresgd} suffered a dependence of $\kappa^d$, we show that the rates of convergence are in fact independent of the dimension $d$.

\begin{thm}\label{thm:bures_gd}
    Assume that $P$ is supported on covariance matrices whose eigenvalues lie in the range $[\lambda_{\min}, \lambda_{\max}]$, $0 < \lambda_{\min} \le \lambda_{\max} < \infty$. Let $\kappa  \deq  \lambda_{\max}/\lambda_{\min}$ denote the condition number. Assume that we initialize at $\Sigma_0 \in \supp P$.
    \begin{enumerate}
        \item (GD) Let $\Sigma_T^{\rm GD}$ denote the $T$-th iterate of GD (Algorithm~\ref{ALG:GD}) with step size $\eta = \frac{1}{2\kappa}$. Then,
        \begin{align*}
            \frac{1}{2\sqrt\kappa} \, W_2^2(\Sigma_T^{\rm GD}, \Sigstar)
            \le F(\Sigma_T^{\rm GD}) - F(\Sigstar)
            &\le \exp\Bigl(- \frac{3T}{64\kappa^{5/2}}\Bigr) \, \{F(\Sigma_0) - F(\Sigstar)\}\,.
        \end{align*}
        \item (SGD) Let $\Sigma_T^{\rm SGD}$ denote the $T$-th iterate of SGD (Algorithm~\ref{ALG:SGD}).
        Then, with appropriately chosen step sizes,
        \begin{align*}
            \frac{1}{2\sqrt \kappa} \E W_2^2(\Sigma_T^{\rm SGD}, \Sigstar)
            \le \E F(\Sigma_T^{\rm SGD}) - F(\Sigstar)
            &\le \frac{48\kappa^3 \var P}{T}\,.
        \end{align*}
    \end{enumerate}
\end{thm}

In fact, using Theorem~\ref{thm:sqrt_lambda_min_concave} we can also relax the conditioning assumption
for gradient descent
to an \emph{average-case} notion of conditioning.
This is a significant improvement when the eigenvalue ranges differ significantly between matrices. 

\begin{thm}\label{thm:bures_avg_case}
    Define the quantities
    \begin{align*}
        \norm{\lambda_{\min}}_{1/2}
        &\deq \Bigl( \int\sqrt{\lambda_{\min}(\Sigma)} \, \D P(\Sigma)\Bigr)^2\,, \\
        \norm{\lambda_{\max}}_1
        &\deq \int \lambda_{\max}(\Sigma) \, \D P(\Sigma)\,.
    \end{align*}
    Then, the conclusions of Theorem~\ref{thm:bures_gd} hold for GD when replacing $\kappa$ with $\norm{\lambda_{\max}}_1/\norm{\lambda_{\min}}_{1/2}$ everywhere\@.\footnote{A previous version of this paper
    stated that the conclusion
    of Theorem~\ref{thm:bures_gd} for SGD also held with averaged eigenvalues, but our proof does not show this.}
\end{thm}

We give the proofs of this result in Appendix~\ref{scn:bures_gd}.

\subsection{Numerical experiments}\label{scn:bures_gd_numerics}

There are two natural competitors of Riemannian GD when minimizing the barycenter functional: $(i)$ solving an SDP (Appendix~\ref{scn:sdp}), and $(ii)$ Euclidean GD (see Appendix~\ref{scn:euclidean_gd} for a description and analysis of the Euclidean gradient descent algorithm).

\begin{figure}[!h]
    \centering
    \begin{minipage}{0.48\textwidth}
        \centering
        \includegraphics[width=\textwidth]{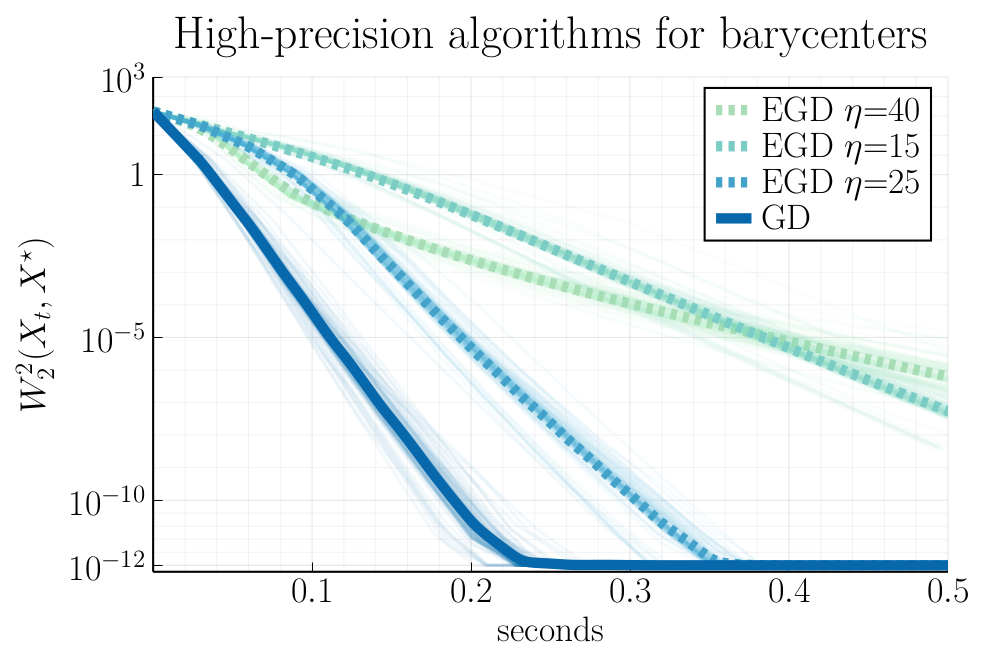}
\caption{
    Riemannian vs.\ Euclidean GD.
    }
        \label{fig:high-precision barycenter}
    \end{minipage}
    \hfill
    \begin{minipage}{0.5\textwidth}
        \centering
        \includegraphics[width=\textwidth]{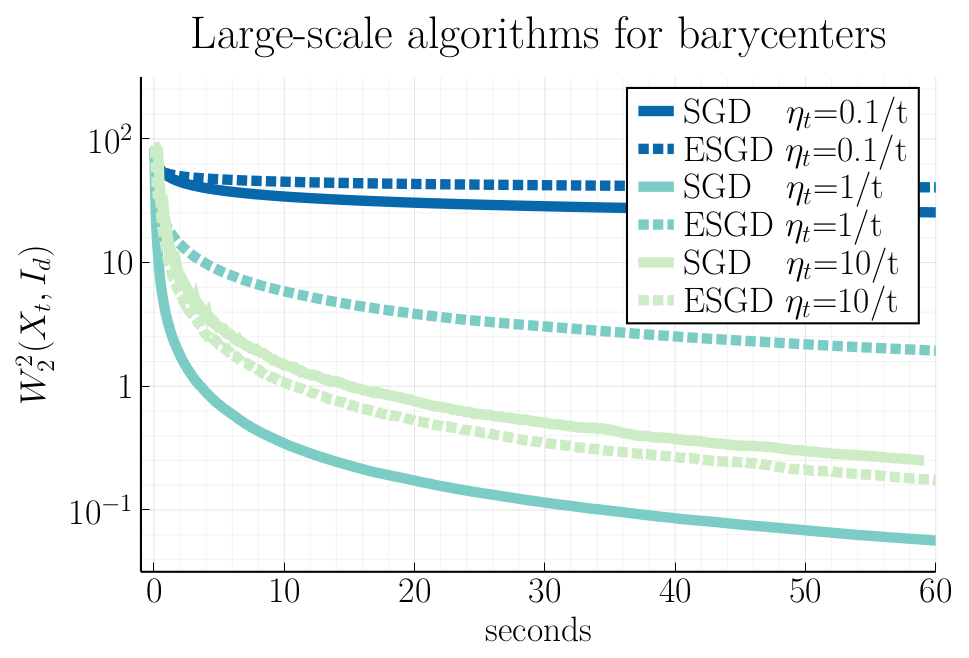}
         \caption{Riemannian vs.\ Euclidean SGD.
         }
         \label{fig:sgd vs esgd}
    \end{minipage}
\end{figure}

In Figure \ref{fig:high-precision barycenter} we compare Riemannian and Euclidean GD on a random dataset consisting of $n=50$ covariance matrices of dimension $d=50$, each with condition number $\kappa = 1000$. The eigenspaces of the matrices are independent Haar distributed, and their eigenvalues are equally spaced in the interval $[\lambda_{\min}, \lambda_{\max}] = [0.03, 30]$. Qualitatively similar results are observed for other input distributions; see Appendix~\ref{scn:experiment_details}. We run $50$ experiments and plot the average accuracy cut off at $10^{-12}$; $X^{\star}$ denotes the best iterate. We omit SDP solvers from the plot because their runtime is orders of magnitude slower for this problem: using the Splitting Cone Solver (SCS) \cite{ocpb:16, scs}, the problem takes ${\sim}15$ seconds to solve, and MOSEK \cite{mosek} is even slower. We observe that Euclidean GD's rate of convergence is very sensitive to its step size, which depends heavily on the conditioning of the problem. Riemannian GD was the clear winner in our experiments, as its step size requires no tuning (following~\cite{zemel2019procrustes}, we simply use step size $\eta = 1$ throughout) and it always performed no worse (in fact, often significantly better) than Euclidean GD. 

In Figure \ref{fig:sgd vs esgd} we compare Riemannian and Euclidean SGD. We average $300 \times 300$ covariance matrices drawn from a distribution whose barycenter is known to be the identity, see Appendix~\ref{scn:experiment_details} for details.
We observe that Riemannian SGD typically outperforms Euclidean SGD, sometimes substantially.

\par We comment on Figure \ref{fig:gd and sgd vs d}, which illustrates the dimension independence of the two Riemannian algorithms, a main result of this paper. It plots the number of passes until convergence $W_2^2(X_t, X^\star) \leq 10^{-r} \var P$ to the barycenter $X^{\star}$, for $r \in \{3, 5\}$. To compare the algorithms on equal footing, the $y$-axis measures ``full passes'' over the $n=50$ matrices: one pass constitutes one iteration of GD, or $n$ iterations of SGD. We generate the input dataset just as in Figure~\ref{fig:high-precision barycenter}. Observe also the tradeoff between GD and SGD: SGD converges rapidly to low-precision solutions, but takes longer to converge to high-precision solutions.

\section{Entropically-regularized barycenters}\label{scn:reg_bary}

In this section, we consider the entropically-regularized barycenter $b_{\rm reg}^\star$ which minimizes 
\begin{align*}
    F_\gamma(b)
    & \deq  \frac{1}{2} \int W_2^2(b, \cdot) \, \D P + \gamma \on{KL}\bigl(b \bigm\Vert \mc N(0, I_d)\bigr)\,,
\end{align*}
where $\KL{\cdot}{\cdot}$ denotes the Kullback-Leibler (KL) divergence, and $\gamma > 0$ is a given regularization parameter. It suffices to consider the case when all of the measures are centered, see Remark~\ref{rmk:noncentered_entropically_reg}. To justify considering this problem on the Bures--Wasserstein space, we provide the following proposition, proven in Appendix~\ref{subsec:ent_unique}.

\begin{prop}\label{prop:ent_unique}
Suppose $P$ is supported on centered Gaussians whose covariance matrices have eigenvalues lying in the range $[1/\sqrt\kappa, \sqrt\kappa]$, for some $\kappa \ge 1$. Then there exists a unique minimizer $b^\star_{\rm reg}$ of $F_{\gamma}$ over $\mc P_2(\R^d)$, and this minimizer is a centered Gaussian distribution whose covariance matrix $\Sigma^\star$ also has eigenvalues in the range $[1/\sqrt\kappa, \sqrt\kappa]$.
\end{prop}

See the introduction for a discussion of the literature on this problem. Prior work focuses on a slightly different entropic penalty, the differential entropy $\int b \ln b$. Note that such a penalty encourages $b$ to be diffuse over all of $\R^d$ (the minimizer blows up as $\gamma \to \infty$). Here, we focus on a KL divergence penalty which has the advantage of interpolating between two well-studied problems: the Wasserstein barycenter problem ($\gamma=0$) and minimization of the KL divergence ($\gamma = \infty$). We take the standard Gaussian distribution as a canonical choice of reference distribution, and note that
our method of analysis can be extended to other
reference measures at the cost of significant
additional technical complexity. We thus choose to exclusively
focus on the standard Gaussian case.

\subsection{Algorithm}

The Riemannian gradient descent algorithm for minimizing $F_\gamma$ is given as Algorithm~\ref{ALG:RGD}. See Appendix~\ref{sec:ent} for a derivation of the update rule.

\begin{algorithm}[H]
  \caption{GD for Regularized Barycenters}\label{ALG:RGD}
  \begin{algorithmic}[1]
    \Procedure{RBary-GD}{$\Sigma_0, P, T, \gamma, \eta$}
      \For{$t = 1, \ldots, T$}
         \State $
        S_t \gets \eta \int \GM(\Sigma_{t-1}^{-1}, \Sigma) \, \D P(\Sigma) + \eta \gamma \Sigma_{t - 1}^{-1}
         + (1 - \eta \,(1 + \gamma))I_d$
        \State $ \Sigma_{t} \gets S_t\Sigma_{t - 1} S_t$
      \EndFor
      \State \Return $\Sigma_{T}$
    \EndProcedure
  \end{algorithmic}
\end{algorithm}

\subsection{Convergence guarantees}

We provide a convergence guarantee for Algorithm~\ref{ALG:RGD}.
We emphasize that as in \S\ref{scn:bary}, our convergence rate is dimension-independent. The proof
appears
in Appendix~\ref{sec:ent}. 

\begin{thm}\label{thm:ent_main} Fix $\gamma > 0$ and suppose that
$P$ is supported on covariance matrices with eigenvalues
in $[1/\sqrt \kappa, \sqrt \kappa]$. If Algorithm~\ref{ALG:RGD} is initialized at a point in $\supp P$
and run with step size $\eta = \frac{2}{{(2+\gamma)}^4\,\kappa}$, then for any $T\ge 1$,
    \begin{align*}
    F_\gamma(\Sigma_T) - F_\gamma(\Sigma^\star)
    &\le \exp\bigl( - \frac{4T}{(2+\gamma)^{20}\,\kappa^5} \bigr) \, \{F_\gamma(\Sigma_0) - F_\gamma(\Sigma^\star)\}\,.
    \end{align*}
\end{thm}

Similarly to Theorem~\ref{thm:bures_gd}, we can also provide guarantees in terms of the distance $W_2(\Sigma_T, \Sigma^\star)$ to the minimizer, as well as guarantees for SGD, but we omit this for brevity.

\subsection{Numerical experiments}

\begin{minipage}{0.48\textwidth}
 In Figure~\ref{fig:effect of regularization}, we investigate the use of the regularization term $\gamma \on{KL}(\cdot \bigm\Vert \mc N(0, I_d))$ to encode a prior belief of isotropy. In Figure \ref{fig:effect of regularization}, we generate $n=100$ i.i.d.\ $20 \times 20$ covariance matrices from a distribution whose barycenter is the identity (see Appendix \ref{scn:experiment_details}). Then, for $\rho \in [0,10]$ we compute the barycenter of a perturbed dataset obtained by adding $\rho e_1 e_1^\T$ to each matrix for different choices of $\gamma$. We see that for $\gamma=0$ the barycenter quickly departs from isotropy, while for larger $\gamma$ the regularization yields averages which are more consistent with our prior belief.
 
 \end{minipage}
 \hfill
 \begin{minipage}{0.48\textwidth}
\includegraphics[width=\textwidth]{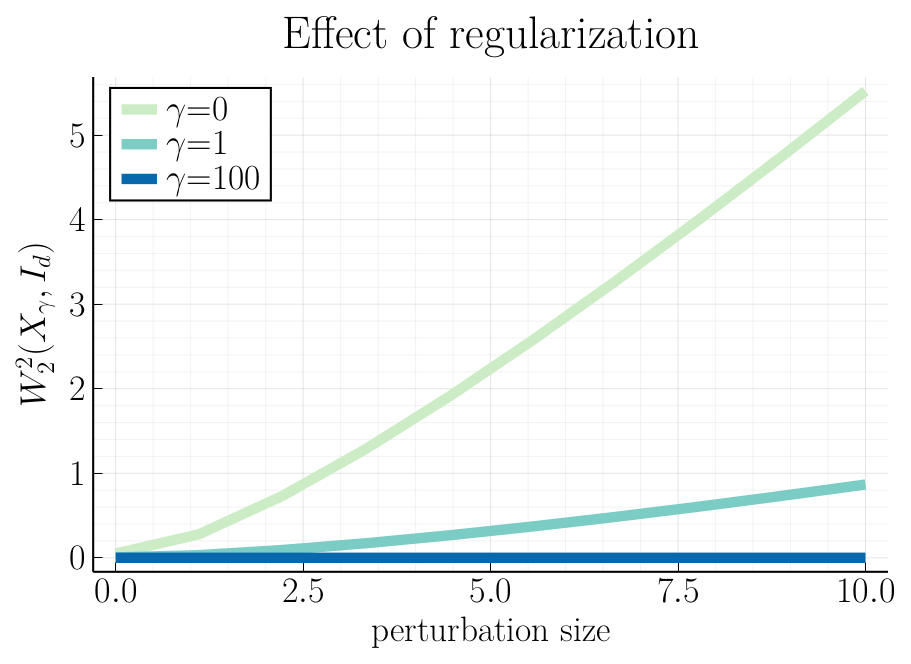}
\captionof{figure}{Effect of regularization for varying $\gamma$. }
        \label{fig:effect of regularization}
\end{minipage}

\section{Geometric medians}\label{scn:median}

In this section, we consider the Wasserstein geometric median
\begin{align}\label{eq:geom_median}
    \bstar_{\rm median} \in \argmin_{b \in \mc P_2(\R^d)} \int W_2(b, \cdot) \, \D P \,.
\end{align}
See the introduction for a discussion of the literature on this problem.
Observe that, in contrast to the barycenter~\eqref{eq:w2_bary}, here we are minimizing the average \emph{unsquared} Wasserstein distance.

The following basic result justifies the consideration of the geometric median problem on the Bures--Wasserstein space.

\begin{prop}\label{prop:basic_prop_median}
    Suppose that $P$ is supported on centered non-degenerate Gaussians whose covariance matrices have eigenvalues upper bounded by $\lambda_{\max}$.
    Then, there exists a solution to~\eqref{eq:geom_median} which is also a centered non-degenerate Gaussian distribution; moreover, its covariance matrix $\Sigstar_{\rm median}$ can be taken to have eigenvalues upper bounded by $\lambda_{\max}$.
\end{prop}
\begin{proof}
    See Appendix~\ref{scn:proofs_median}.
\end{proof}

\begin{rema}
    Suppose now that $P$ is supported on non-degenerate Gaussian distributions which are not necessarily centered. Then, the proof of Proposition~\ref{prop:basic_prop_median} applies with minor modifications to show that the minimizer of the median functional is still attained at a Gaussian distribution.
    However, unlike the barycenter and entropically regularized barycenter, it is not the case that the mean of the Wasserstein geometric median is the Euclidean geometric median of the means, thus it is not as straightforward to reduce to the centered case for this problem. Nevertheless, in Appendix~\ref{scn:non_centered_median}, we describe a reduction which allows the algorithm described in the next section to be applied in a black box manner to the non-centered case, with corresponding convergence guarantees.
\end{rema}

\subsection{Algorithm}

Since the Wasserstein distance $W_2(\Sigma,\cdot)$ is neither geodesically convex nor geodesically smooth, nor Euclidean convex nor Euclidean smooth (see Remark~\ref{rmk:w2_horrible}), it poses challenges for optimization.
We therefore smooth the objective before optimization.
Given a desired target accuracy $\varepsilon > 0$, let
\begin{align*}
    W_{2,\varepsilon}  \deq  \sqrt{W_2^2 + \varepsilon^2}\,, \qquad F_\varepsilon(b)  \deq  \int W_{2,\varepsilon}(b, \cdot) \, \D P\,.
\end{align*}
The smoothed Riemannian gradient descent algorithm is given as Algorithm~\ref{ALG:smoothed}. See Appendix~\ref{app:median} for a derivation of the update rule.

\begin{algorithm}[H]
  \caption{Smoothed GD for Median}\label{ALG:smoothed}
  \begin{algorithmic}[1]
    \Procedure{Median-GD}{$\Sigma_0, P, T, \varepsilon, \eta$}
      \For{$t = 1, \ldots, T$}
        \State $
        S_t \gets I_d + \eta \int \{\GM(\Sigma_{t - 1}^{-1}, \Sigma) - I_d\} \, {W_{2,\varepsilon}(\Sigma_{t-1}, \Sigma)}^{-1} \, \D P(\Sigma)
        $
        \State $ \Sigma_{t} \gets S_t\Sigma_{t - 1} S_t$
      \EndFor
      \State \Return $\Sigma_{T}$
    \EndProcedure
  \end{algorithmic}
\end{algorithm}

\subsection{Convergence guarantees}

Despite the smoothing, the objective $F_\varepsilon$ is still non-convex, and we do not provide a global minimization guarantee. Instead,
we now show that Algorithm~\ref{ALG:smoothed} can find an $\mc O(\varepsilon)$-stationary point for the smoothed geometric median functional $F_\varepsilon$ in $\mc O(1/\varepsilon^3)$ iterations. We emphasize that as in our other results, this convergence rate is dimension-independent.

\begin{thm}\label{thm:median_guarantee}
    Assume that we initialize Algorithm~\ref{ALG:smoothed} at $\Sigma_0 \in \supp P$ with step size $\eta = \varepsilon$.
    Then, Algorithm~\ref{ALG:smoothed}
    yields iterates with $\min_{t=0,1,\dotsc,T}{\norm{\nabla F_\varepsilon(\Sigma_t)}_{\Sigma_t}} \le \varepsilon$ if
    \begin{align*}
        T
        &\ge \frac{2 F_\varepsilon(\Sigma_0)} {\varepsilon^3}\,.
    \end{align*}
\end{thm}
\begin{proof}
    See Appendix~\ref{scn:proofs_median}.
\end{proof}

An earlier version
of Theorem~\ref{thm:median_guarantee}
gave guarantees
on the objective
gap
$F(\Sigma_T)
- F(\Sigma^\star_{\rm median})$. However,
the proof of this
result contained a mistake. We have thus
changed Theorem~\ref{thm:median_guarantee}
to guarantee only
approximate
stationarity,
and leave a study
of global convergence
for the Bures--Wasserstein
median functional
to future work.

\subsection{Numerical experiments}
\begin{figure}[!h]
    \centering
    \begin{minipage}{.5\textwidth}
        \centering
         \includegraphics[width=\textwidth]{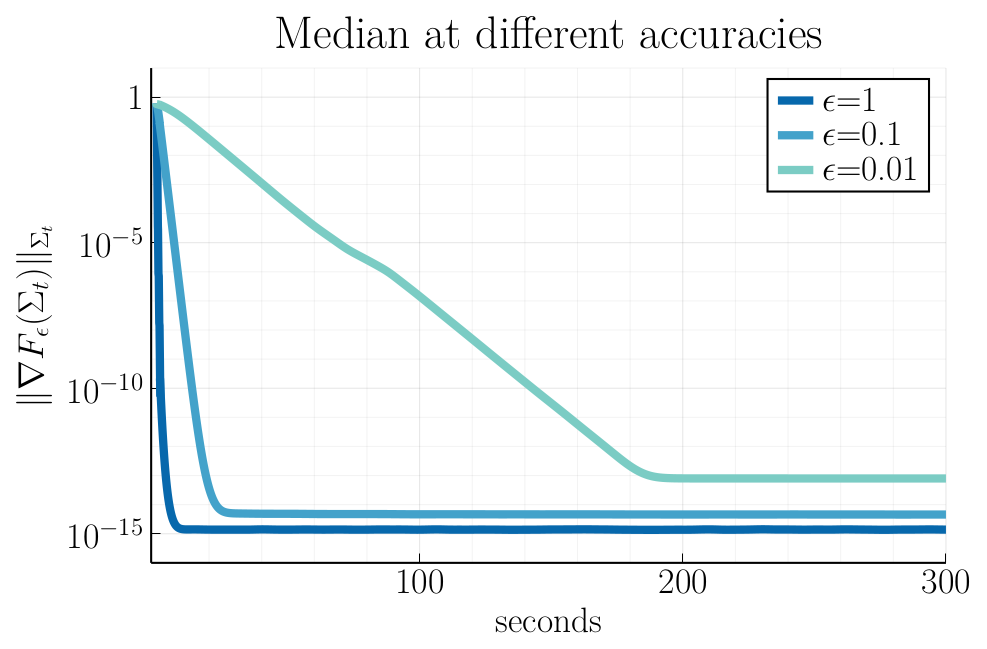}
\caption{Evolution of the gradient of the median objective for varying $\varepsilon$. }
        \label{fig:median precision}
    \end{minipage}
    \hfill
    \begin{minipage}{0.48\textwidth}
        \centering
        \includegraphics[width=\textwidth]{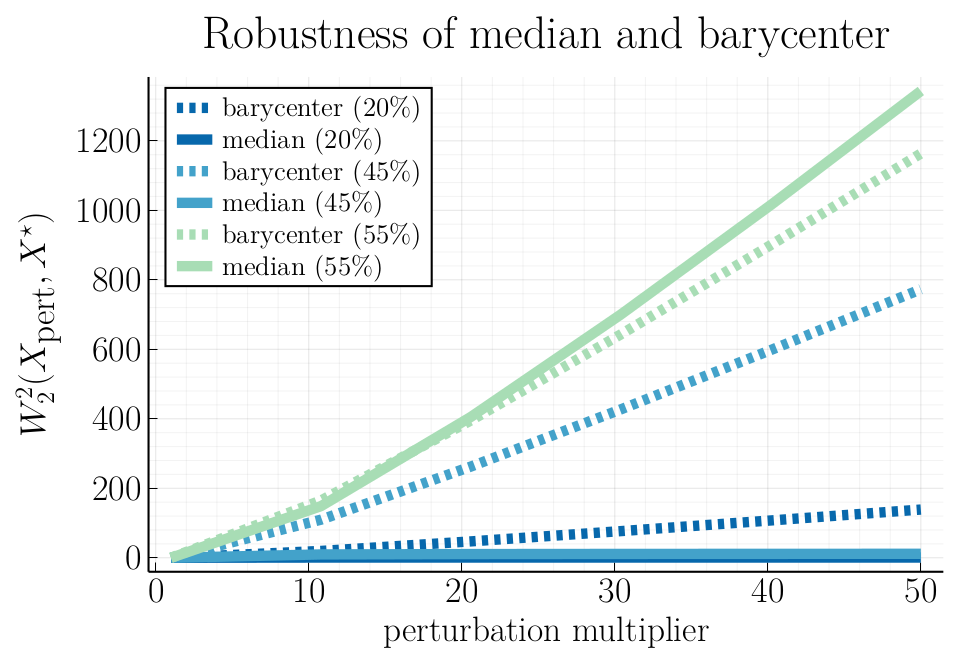}
        \caption{Robustness of the Wasserstein median.}
        \label{fig:robustness}
    \end{minipage}
\end{figure}

In Figure \ref{fig:median precision} we plot the gradient of the smoothed median functional $\|\nabla F_\varepsilon(\Sigma_t)\|_{\Sigma_t}$ as we optimize $F_\varepsilon$ using Algorithm~\ref{ALG:smoothed} for varying choices of $\varepsilon$ for different values of $\varepsilon$. There is a trade-off for choosing the regularization parameter $\varepsilon$: as $\varepsilon$ increases the rate of convergence to a stationary point becomes slower. The covariance matrices are generated as in Figure \ref{fig:high-precision barycenter}, with $n=d=30$ and $[\lambda_{\min}, \lambda_{\max}] = [0.01,10]$. 

In Figure \ref{fig:robustness} we illustrate the robustness of the Wasserstein geometric median up to its breakdown point of $50\%$ \cite{FKJ}. We take random input matrices as above, with $n=d=20$ and $[\lambda_{\min}, \lambda_{\max}] = [1,10]$, and compute their barycenter and approximate median ($\varepsilon=1$). We then perturb a fraction ($20\%$, $45\%$, and $55\%$ for our figure) of the matrices by multiplying them by a constant greater than $1$. The $x$-axis of the plot shows the size of the perturbation while the $y$-axis gives the distance of the original barycenter and median to the barycenter and median of this new, perturbed dataset.

We also implemented Euclidean GD for this geometric median problem; plots are omitted for brevity since the results are similar to those for the barycenter (c.f.\ Section~\ref{scn:bures_gd_numerics}) in that Euclidean GD depends much more heavily on parameter tuning. Note also that Euclidean GD does not come with global convergence guarantees for this problem since it is non-convex in the Euclidean geometry.

\bigskip

\noindent{\bf Acknowledgments.} \\
We are grateful to Victor-Emmanuel Brunel, Tyler Maunu, Pablo Parrilo, and Philippe Rigollet for stimulating conversations, and particularly to Aleksey Kroshnin for bringing the flaw in the previous version of Theorem~\ref{thm:bures_avg_case} to our attention. We also thank Pablo Parrilo for pointing out that Bures--Wasserstein barycenters have an SDP formulation (Appendix~\ref{scn:sdp}).

JA was supported by NSF Graduate Research Fellowship 1122374 and a TwoSigma PhD fellowship.
SC and AS were supported by the Department of Defense (DoD) through the National Defense Science \& Engineering Graduate Fellowship (NDSEG) Program.

\newpage
\appendix
\part{Appendix}
\parttoc[e]

\section{Background on the Bures--Wasserstein manifold}\label{scn:bw_background}

In this section, we collect relevant background about Bures--Wasserstein geometry to make the paper more self-contained.

\subsection{Geometry}\label{scn:bures_geometry}

We begin by describing the geometry of optimal transport, and then explain how to specialize the general concepts to the Bures--Wasserstein manifold. The books~\cite{ambrosio2008gradient, villani2009ot} are definitive references for the Riemannian structure of optimal transport. We attempt to convey the main relevant ideas, and in doing so do not attempt to be fully rigorous here.

Let $\mc P_{2,\rm ac}(\R^d)$ denote the space of all probability measures on $\R^d$ which are absolutely continuous (i.e., admit a density w.r.t.\ the Lebesgue measure) and which have a finite second moment. When equipped with the $2$-Wasserstein distance $W_2$, it becomes a metric space.
In fact, more is true: $(\mc P_{2,\rm ac}(\R^d), W_2)$ admits a formal Riemannian structure which we now describe. Given $\mu_0, \mu_1 \in \mc P_{2,\rm ac}(\R^d)$, let $T$ denote the optimal transport map from $\mu_0$ to $\mu_1$; thus, $T : \R^d\to\R^d$ is a map satisfying $T_\# \mu_0 = \mu_1$. Here, $\#$ denotes the pushforward operation, i.e., if $X \sim \mu_0$, then $T(X) \sim \mu_1$. The constant-speed geodesic ${(\mu_t)}_{t\in [0,1]}$ joining $\mu_0$ to $\mu_1$ is described via
\begin{align*}
    \mu_t = {[(1-t)\, {\id} + tT]}_\# \mu_0\,, \qquad t \in [0,1]\,.
\end{align*}
This geodesic has the following interpretation: draw a ``particle'' $X_0 \sim \mu_0$, and move $X_0$ to $T(X_0)$ with constant speed for one unit of time along the Euclidean geodesic (i.e., straight line) joining these endpoints; thus, at time $t$, the particle is at position $X_t = (1-t) X_0 + tT(X_0)$. Then, $\mu_t$ is simply the law of $X_t$.

We take the tangent vector of the geodesic ${(\mu_t)}_{t\in [0,1]}$ at time $0$ to be the mapping $T - {\id}$; note that in the particle view, $T(X_0) - X_0$ represents the velocity of the particle at time $0$. The tangent space $T_{\mu_0} \mc P_{2,\rm ac}(\R^d)$ to $\mc P_{2,\rm ac}(\R^d)$ at $\mu_0$ is then defined to consist of all possible tangent vectors to geodesics emanating from $\mu_0$. Actually, in order to make $T_{\mu_0} \mc P_{2,\rm ac}(\R^d)$ formally into a (closed subset of a) Hilbert space, the definition is modified to read (\cite[Theorem 8.5.1]{ambrosio2008gradient})
\begin{align}\label{eq:w2_geod}
    T_{\mu_0} \mc P_{2,\rm ac}(\R^d)
    & \deq  \overline{\{\lambda \, (T_{\mu_0\to\nu} - {\id}) : \lambda > 0, \; \nu \in \mc P_{2,\rm ac}(\R^d)\}}^{L^2(\mu_0)}\,,
\end{align}
where the overline denotes the $L^2(\mu_0)$ closure; we equip this tangent space with the $L^2(\mu_0)$ norm. Thus, for instance, we have $W_2^2(\mu_0, \mu_1) = \E[\norm{X_0 - T_{\mu_0\to\mu_1}(X_0)}^2] = \norm{T_{\mu_0\to\mu_1} - {\id}}_{L^2(\mu_0)}^2$, which says that the squared norm of the tangent vector of the geodesic ${(\mu_t)}_{t\in [0,1]}$ equals the squared Wasserstein distance. We may write $\norm \cdot_{\mu_0}$ as a shorthand for $\norm\cdot_{L^2(\mu_0)}$.

The Riemannian exponential map $\exp_\mu$ is the mapping $T_\mu \mc P_{2,\rm ac}(\R^d) \to \mc P_{2,\rm ac}(\R^d)$ which maps a tangent vector $v$ to the constant-speed geodesic emanating from $\mu$ with velocity $v$, evaluated at time $1$.\footnote{Generally, in Riemannian geometry, the exponential map is not defined on the entire tangent space but rather a subset of it; this is also the case for Wasserstein space.} From our description above, we see that $\exp_\mu v = {({\id} + v)}_\# \mu$, since the tangent vector joining $\mu$ to $T_\# \mu$ is $v = T - {\id}$ (when $T$ is an optimal transport map).
It is also convenient to define the Riemannian logarithmic map $\log_\mu : \mc P_{2,\rm ac}(\R^d) \to T_\mu \mc P_{2,\rm ac}(\R^d)$ to be the inverse of the exponential map $\exp_\mu$; in our context, $\log_\mu \nu = T_{\mu\to\nu} - {\id}$.

In Riemannian geometry, it is common to localize the argument around a measure $\mu$, which loosely means replacing a measure $\nu$ with its image $\log_\mu \nu$ in the tangent space at $\mu$. This is convenient because the tangent space at $\mu$ is embedded in the Hilbert space $L^2(\mu)$, and we can leverage Hilbert space arguments (e.g., computing inner products). In order to do this one must quantify the distortion introduced by the map $\log_\mu$, which is morally related to the curvature of the manifold.

We now specialize the above concepts to the Bures--Wasserstein manifold, in which non-degenerate centered Gaussians are identified with their covariance matrices; thus, the Bures--Wasserstein manifold is the space $\psd$ of positive-definite symmetric matrices equipped with a certain Riemannian metric.

The optimal transport problem between Gaussians is discussed in many places, e.g.,~\cite{bhatiajainlim2019bures}. Given two covariance matrices $\Sigma,\Sigma'\in\psd$, the optimal transport map between the corresponding centered Gaussians is the linear map $\R^d\to\R^d$ given by
\begin{align*}
    T_{\Sigma\to\Sigma'}
    &= \Sigma^{-1/2} \, {(\Sigma^{1/2} \Sigma' \Sigma^{1/2})}^{1/2} \, \Sigma^{-1/2}\,.
\end{align*}
Note that this is a symmetric matrix. Since $AX \sim \mc N(0, A\Sigma A^\T)$ for $X \sim \mc N(0, \Sigma)$, the fact that $T_{\Sigma\to\Sigma'} X \sim \mc N(0,\Sigma')$ reduces to the matrix identity $T_{\Sigma\to\Sigma'} \Sigma T_{\Sigma\to\Sigma'} = \Sigma'$, which can be verified by hand. The above formula yields
\begin{align}\label{eq:alternate_w2_gaussian}
\begin{aligned}
    W_2^2(\Sigma, \Sigma')
    &= \E[\norm{X - T_{\Sigma\to\Sigma'} X}^2]
    = \E[\norm X^2 + \norm{T_{\Sigma\to\Sigma'} X}^2 - 2 \, \langle X, T_{\Sigma\to\Sigma'} X\rangle] \\
    &= \tr(\Sigma + \Sigma' - 2\Sigma T_{\Sigma\to\Sigma'})\,.
\end{aligned}
\end{align}

From the general description of Wasserstein geodesics, the constant-speed geodesic ${(\Sigma_t)}_{t\in [0,1]}$ joining $\Sigma$ to $\Sigma'$ is given by
\begin{align}\label{eq:bures_geod}
    \Sigma_t
    &= \bigl((1-t) I_d + t T_{\Sigma\to\Sigma'}\bigr) \Sigma \bigl((1-t) I_d + t T_{\Sigma\to\Sigma'}\bigr) \,, \qquad t \in [0,1]\,.
\end{align}
The tangent space $T_\Sigma \psd$ is identified with the space $\sym$ of symmetric $d\times d$ matrices. Given $S \in T_\Sigma \psd$, the tangent space norm of $S$ is given by $\norm S_{L^2(\mc N(0, \Sigma))} = \sqrt{\E[\norm{SX}^2]} = \sqrt{\langle S^2,\Sigma\rangle}$, which we simply denote as $\norm S_\Sigma$. More generally, given matrices $A$, $B$, we write $\langle A, B\rangle_\Sigma  \deq  \tr(A^\T\Sigma B)$. The exponential map\footnote{Technically the exponential map is only defined if $S + I_d \succeq 0$; this is because if $S + I_d$ is not positive semidefinite, then $S + I_d$ is not an optimal transport map due to Brenier's theorem.} is $\exp_\Sigma S = (I_d + S)\,\Sigma\, (I_d + S)$, so that $\exp_\Sigma(T_{\Sigma\to\Sigma'} - I_d) = \Sigma'$. The inverse of the exponential map is then $\log_\Sigma \Sigma' = T_{\Sigma\to\Sigma'} - I_d$.

The description of the Bures--Wasserstein tangent space is in accordance with the general Riemannian structure of Wasserstein space (see~\cite{ambrosio2008gradient}) and agrees with the convention in~\cite{chewietal2020buresgd}. We now elaborate on other possible conventions, in order to dispel possible confusion.

The space $\psd$ is often studied as a manifold in other contexts, and the tangent space at any point is usually identified with $\sym$. It is crucial to realize, however, that a tangent space is not simply a vector space (or inner product space); a tangent space also has the interpretation of describing velocities of curves. In other words, for each tangent vector $S$, we also need to prescribe which curves have velocity $S$. In the usual way of describing the manifold structure of $\psd$, this prescription is given as follows. Given a curve ${(\Sigma_t)}_{t\in\R} \subseteq \psd$, if $\dot\Sigma_0$ denotes the ordinary time derivative of this curve at time $0$, then we declare $\dot\Sigma_0$ to be the tangent vector of the curve at time $0$. Although this prescription is natural, observe that it conflicts with our description of the tangent space structure of the Bures--Wasserstein manifold; in particular, for the curve in~\eqref{eq:bures_geod}, we have described the tangent vector to this curve (at time $0$) to be $T_{\Sigma\to\Sigma'} - I_d$, but the ordinary time derivative of this curve is $(T_{\Sigma\to\Sigma'} - I_d) \Sigma + \Sigma (T_{\Sigma\to\Sigma'} - I_d)$.

To summarize the discussion in the preceding paragraph: although the usual description of the tangent space of $\psd$ at $\Sigma$ and our description of the tangent space are formally the same, in that they are both identified with $\sym$, they differ in that tangent vectors from the two descriptions give rise to different curves. Note that if we were to adopt the usual description of the tangent space of $\psd$, then we would have to define the tangent space norm $\norm\cdot_\Sigma$ differently from above. In this paper, we adopt the convention described earlier in this section in order to preserve the connection with the general setting of optimal transport.

\subsection{Geodesic convexity and generalized geodesic convexity}\label{scn:gen_geod_cvxty}

Once we have geodesics, we can then define convex functions. A function $f : \mc P_{2,\rm ac}(\R^d) \to \R$ is said to be \emph{geodesically convex} if for all constant-speed geodesics ${(\mu_t)}_{t\in [0,1]}$ (i.e., curves described by~\eqref{eq:w2_geod}), it holds that
\begin{align}\label{eq:geod_cvx_fn}
    f(\mu_t)
    &\le (1-t)\, f(\mu_0) + t \, f(\mu_1)\,, \qquad t\in [0,1]\,.
\end{align}
It turns out, however, that many natural examples of geodesically convex functions on Wasserstein space are convex in a stronger sense, in that they satisfy the inequality~\eqref{eq:geod_cvx_fn} for a larger class of curves than geodesics.
A \emph{generalized geodesic} from $\mu_0$ to $\mu_1$, with basepoint $\nu \in \mc P_{2,\rm ac}(\R^d)$, is defined to be the curve ${(\mu_t^\nu)}_{t\in [0,1]}$ where
\begin{align*}
    \mu_t^\nu
    & \deq  {[(1-t) T_{\nu\to\mu_0} + t T_{\nu\to\mu_1}]}_\# \nu\,, \qquad t \in [0,1]\,.
\end{align*}
A function $f : \mc P_{2,\rm ac}(\R^d) \to \R$ is said to be \emph{convex along generalized geodesics} if for every generalized geodesic ${(\mu_t^\nu)}_{t\in [0,1]}$, it holds that
\begin{align*}
    f(\mu_t^\nu)
    &\le (1-t)\, f(\mu_0) + t \, f(\mu_1)\,, \qquad t\in [0,1]\,.
\end{align*}
Note that the geodesic ${(\mu_t)}_{t\in [0,1]}$ joining $\mu_0$ to $\mu_1$ coincides with the generalized geodesic ${(\mu_t^{\mu_0})}_{t\in [0,1]}$, so that convexity along generalized geodesics is indeed stronger than geodesic convexity.

Generalized geodesics were studied in~\cite{ambrosio2008gradient} in order to rigorously study gradient flows on Wasserstein space. The added flexibility of generalized geodesics is sometimes important for applications~\cite{ahnchewi2021mirrorlangevin}; in our work, as well as in~\cite{chewietal2020buresgd}, generalized geodesics are needed to study the iterates of Riemannian gradient descent.

The interpretation of generalized geodesics is that we linearize $\mc P_{2,\rm ac}(\R^d)$ on the tangent space $T_\nu \mc P_{2,\rm ac}(\R^d)$. This means that we replace $\mu_0$ with its image $\log_\nu \mu_0 = T_{\nu\to\mu_0} - {\id}$ in the tangent space, and similarly for $\mu_1$. Since the tangent space is a subset of a Hilbert space, geodesics in the tangent space are described by straight lines, i.e.,
\begin{align*}
    t \mapsto (1-t) T_{\nu\to\mu_0} + t T_{\nu\to\mu_1} - {\id}\,.
\end{align*}
If we translate back to $\mc P_{2,\rm ac}(\R^d)$, we end up with the curve
\begin{align*}
    t \mapsto \exp_\nu\bigl((1-t) T_{\nu\to\mu_0} + t T_{\nu\to\mu_1} - {\id}\bigr)
    = {[(1-t) T_{\nu\to\mu_0} + t T_{\nu\to\mu_1}]}_\# \nu
    = \mu_t^\nu\,.
\end{align*}
Thus, the property of being convex along generalized geodesics can be reformulated as requiring that
\begin{align}\label{eq:alternate_gen_geod}
    f \circ \exp_\nu : T_\nu \mc P_{2,\rm ac}(\R^d)\to\R \qquad\text{is convex for every}~\nu \in \mc P_{2,\rm ac}(\R^d)\,.
\end{align}

In Euclidean space, convexity of a function $f : \R^d\to\R$ is equivalent, via Jensen's inequality, to the following statement: for every probability measure $P$ on $\R^d$, it holds that $f(\int x \, \D P(x)) \le \int f(x) \, \D P(x)$. Since the Wasserstein barycenter is the Wasserstein analogue of the mean, we can write a similar definition on Wasserstein space.
Given a probability measure $P$ on $\mc P_{2,\rm ac}(\R^d)$, let $b_P$ denote its Wasserstein barycenter. We say that $f : \mc P_{2,\rm ac}(\R^d) \to \R$ is \emph{convex along barycenters} if
\begin{align*}
    f(b_P)
    &\le \int f(\mu) \, \D P(\mu)\,, \qquad \text{for all}~P \in \mc P_2\bigl(\mc P_{2,\rm ac}(\R^d)\bigr)\,.
\end{align*}
Similarly, via~\eqref{eq:alternate_gen_geod}, we can define $f : \mc P_{2,\rm ac}(\R^d) \to \R$ to be convex along generalized barycenters if
\begin{align}\label{eq:conv_along_gen_bary}
\begin{aligned}
    f\circ \exp_\nu\Bigl(\int v \, \D P(v)\Bigr) &\le \int f\circ \exp_\nu(v) \, \D P(v) \\
    &\qquad\qquad{} \text{for all}~\nu \in \mc P_{2,\rm ac}(\R^d)~\text{and}~P \in \mc P_2\bigl(T_\nu\mc P_{2,\rm ac}(\R^d)\bigr)\,.
    \end{aligned}
\end{align}
However, since the tangent space is embedded in a Hilbert space, there is no difference between~\eqref{eq:alternate_gen_geod} and~\eqref{eq:conv_along_gen_bary}.

To summarize the relationship between these four concepts:
\begin{align*}
    \text{convex along generalized barycenters}
    &\iff \text{convex along generalized geodesics} \\
    &\implies \text{convex along barycenters}
    \implies \text{geodesically convex}\,.
\end{align*}
For a justification of these facts and further discussion, see~\cite{agueh2011barycenter}.

\subsection{Geodesic optimization}\label{scn:geod_opt}

Given a functional $F : \mc P_{2,\rm ac}(\R^d)\to\R$, we can define its Wasserstein gradient formally as follows.
For any constant-speed geodesic ${(\mu_t)}_{t\in [0,1]}$, the gradient of $F$ at $\mu_0$ is the element $\nabla F(\mu_0) \in T_{\mu_0} \mc P_{2,\rm ac}(\R^d)$ satisfying
\begin{align*}
    \partial_t|_{t=0} F(\mu_t)
    &= \langle \nabla F(\mu_0), T_{\mu_0\to\mu_1} - {\id}\rangle_{\mu_0}\,.
\end{align*}
The \emph{Riemannian gradient descent} update for $F$ with step size $\eta$ starting at $\mu$ is
\begin{align*}
    \mu^+
    & \deq  \exp_\mu\bigl(-\eta \nabla F(\mu)\bigr)
    = {[{\id} - \eta \nabla F(\mu)]}_\# \mu\,.
\end{align*}
Note that the step size $\eta$ should be chosen small enough that $-\eta \nabla F(\mu)$ lies in the domain of the exponential map. 
From the general description of the tangent space of Wasserstein space, $\nabla F(\mu)$ is the gradient of a mapping $\psi : \R^d\to\R$; then, $-\eta \nabla F(\mu)$ belongs to the domain of the exponential map if $\norm\cdot^2/2 - \eta \psi$ is convex.

We say that $F$ is $\alpha$-strongly convex if
\begin{align*}
    F(\mu_1)
    &\ge F(\mu_0) + \langle \nabla F(\mu_0), \log_{\mu_0} \mu_1 \rangle_{\mu_0} + \frac{\alpha}{2} \, W_2^2(\mu_0,\mu_1)\,, \qquad\text{for all}~\mu_0,\mu_1 \in \mc P_{2,\rm ac}(\R^d)\,,
\end{align*}
and $\beta$-smooth if
\begin{align*}
    F(\mu_1)
    &\le F(\mu_0) + \langle \nabla F(\mu_0), \log_{\mu_0} \mu_1 \rangle_{\mu_0} + \frac{\beta}{2} \, W_2^2(\mu_0,\mu_1)\,, \qquad\text{for all}~\mu_0,\mu_1 \in \mc P_{2,\rm ac}(\R^d)\,.
\end{align*}
These two properties are formally equivalent to the following statements: for any constant-speed geodesic ${(\mu_t)}_{t\in [0,1]}$, one has
\begin{align*}
    \partial_t^2|_{t=0} F(\mu_t)
    &\ge \alpha \, W_2^2(\mu_0, \mu_1)
    \qquad\text{or}\qquad\partial_t^2|_{t=0} F(\mu_t)
    \le \beta \, W_2^2(\mu_0, \mu_1),
\end{align*}
respectively.

\subsection{Curvature and the barycenter functional}\label{scn:curv_bary}

One of the interesting features of the barycenter problem is that, because it is defined in terms of the squared distance function, it captures key geometric features of the underlying space; in fact, this is arguably the reason for the success of the barycenter for geometric applications. To further discuss this connection, it is insightful to abstract the situation to computing barycenters on a metric space. 

Given a metric space $(X, d)$ and a probability measure $P$ on $X$, a barycenter of $P$ is a solution of
\begin{align*}
    \minimize_{b\in X} \qquad F_P(b)  \deq  \frac{1}{2} \int d^2(b, \cdot) \, \D P\,.
\end{align*}
The basic structure required on $X$ in order to study first-order optimization methods is the presence of geodesics. This is formalized by the notion of a \emph{geodesic space}, which is studied in metric geometry; see~\cite{buragoivanov2001metricgeometry}. Then, we may define a function $F : X \to\R$ to be \emph{$\alpha$-strongly convex} if for all geodesics ${(x_t)}_{t\in [0,1]}$ in $X$, it holds that
\begin{align*}
    F(x_t)
    &\le (1-t) \, F(x_0) + t \, F(x_1) - \frac{\alpha \, t \, (1-t)}{2} \, d^2(x_0,x_1)\,, \qquad \text{for all}~t \in [0,1]\,.
\end{align*}

It is known that the convexity properties of the barycenter functional $F_P$ are related to the \emph{curvature} of the space. Here, curvature is interpreted as the \emph{Alexandrov curvature}, which is the generalization of sectional curvature to geodesic spaces, see~\cite{buragoivanov2001metricgeometry}. Then, the result is that $F_P$ is $1$-strongly convex for every probability measure $P$ on $X$ if and only if $X$ has \emph{non-positive curvature}; see~\cite{sturm2003npc} for precise statements. In fact, the $1$-strong convexity of barycenter functionals is essentially the definition of non-positive curvature in this context.

Consequently, much stronger results are known for barycenters in non-positively curved spaces, ranging from basic properties such as existence and uniqueness, to statistical estimation and optimization; for details see the nice article~\cite{sturm2003npc}.

In contrast, it is well-known that Wasserstein space $\mc P_{2,\rm ac}(\R^d)$ (and hence, the Bures--Wasserstein space) is \emph{non-negatively curved}~\cite[Theorem 7.3.2]{ambrosio2008gradient}. This means, for instance, that convexity and properties related to convexity (such as the PL inequality employed in Appendix~\ref{scn:bures_gd}) are not automatic for the barycenter functional in Wasserstein space. On the other hand, as emphasized in~\cite{chewietal2020buresgd}, this non-negative curvature is related to the \emph{smoothness} of the barycenter functional.

\subsection{Additional facts about the Wasserstein metric}\label{scn:facts}

Here we collect various facts about the Wasserstein metric for easy reference in the sequel. 

\begin{enumerate}
    \item \textbf{Euclidean gradient vs.\ Bures--Wasserstein gradient}.\label{fact:bures_gradient}
    
    Let $F : \psd \to \R$ be a function. Throughout this paper, we denote by $\on D F$ the usual Euclidean gradient of $F$, and we reserve $\nabla F$ for the gradient with respect to the Bures--Wasserstein geometry. In fact, under our tangent space convention, these two quantities are related as follows: let ${(\Sigma_t)}_{t\in\R}$ denote a curve in $\psd$. We temporarily denote the Euclidean tangent vector (i.e., ordinary time derivative) to this curve via $\dot \Sigma^{\rm E}$, and the Bures--Wasserstein tangent vector via $\dot \Sigma^{\rm BW}$, which are related via $\dot\Sigma^{\rm E} = \dot\Sigma^{\rm BW}\Sigma + \Sigma \dot \Sigma^{\rm BW}$ (see the discussion in Appendix~\ref{scn:bures_geometry}). We can compute the time derivative of $F$ in two ways:
    \begin{align*}
        \langle \nabla F(\Sigma_0), \dot \Sigma^{\rm BW}_0\rangle_{\Sigma_0}
        &= \partial_t|_{t=0} F(\Sigma_t)
        = \langle \on D F(\Sigma_0), \dot \Sigma_0^{\rm E} \rangle
        = \langle \on D F(\Sigma_0), \dot \Sigma^{\rm BW}_0 \Sigma_0 + \Sigma_0 \dot \Sigma^{\rm BW}_0 \rangle \\
        &= 2 \,\langle \on D F(\Sigma_0), \dot \Sigma^{\rm BW}_0 \rangle_{\Sigma_0}\,.
    \end{align*}
    From this we can conclude that
    \begin{align*}
        \nabla F(\Sigma_0)
        &= 2\,\on D F(\Sigma_0)\,.
    \end{align*}
    \item \textbf{Gradient of the squared Wasserstein distance}.\label{fact:grad_sq_dist}
    
    For any $\nu\in \mc P_{2,\rm ac}(\R^d)$, the gradient of the functional $W_2^2(\cdot,\nu)$ at $\mu$ is given by
    \begin{align*}
        \nabla W_2^2(\cdot, \nu)(\mu)
        &= -2\,(T_{\mu\to\nu} - {\id}) = -2\log_\mu\nu\,.
    \end{align*}
    This is derived in, e.g.,~\cite{zemel2019procrustes}. In the Bures--Wasserstein setting, it can be proven via matrix calculus; see the proof of Theorem~\ref{lem:strong convexity smoothness}.
    \item \textbf{Inverse of the transport map}.\label{fact:inv_transport_map}
    
    If $\Sigma,\Sigma' \in \psd$, then the transport map $T_{\Sigma\to\Sigma'}$ is the inverse matrix for the transport map $T_{\Sigma'\to\Sigma}$. This can be verified from the formula~\eqref{eq:gaussian_ot_map} using the symmetry of the geometric mean.
    More generally, it is a special case of the convex conjugacy relation between optimal Kantorovich potentials.
    \item \textbf{Diagonal case}.\label{fact:diagonal}
    
    If $\Sigma_0, \Sigma_1 \in \psd$ are \emph{diagonal matrices}, then $W_2^2(\Sigma_0,\Sigma_1) = \norm{\Sigma_0^{1/2} - \Sigma_1^{1/2}}_{\rm F}^2$ is the squared Frobenius norm between the square roots. This can be verified, e.g., from the explicit formula~\eqref{eq:w2_formula} using the fact that $\Sigma_0$ and $\Sigma_1$ commute. Note that in one dimension, all matrices are diagonal. More generally, these observations extend to when $\Sigma_0$ and $\Sigma_1$ commute.
    
    Similarly, it can be seen from~\eqref{eq:bures_geod} that the geodesic is given by
    \begin{align*}
        \Sigma_t^{1/2}
        &= (1-t) \, \Sigma_0^{1/2} + t \, \Sigma_1^{1/2}\,, \qquad t \in [0,1]\,,
    \end{align*}
    which says that the Bures--Wasserstein geodesic between diagonal (or commuting matrices) is simply the Euclidean geodesic after applying the square root map.
    \item \textbf{The case of non-zero means}. \label{fact:non_zero_mean}
    
    For any $\mu,\nu \in \mc P_2(\R^d)$, suppose that the means of these distributions are $m_\mu$ and $m_\nu$, respectively.
    Let $\bar\mu$, $\bar\nu$ denote the centered versions of these distributions.
    Then, it holds that
    \begin{align*}
        W_2^2(\mu,\nu)
        &= \norm{m_\mu - m_\nu}^2 + W_2^2(\bar\mu,\bar\nu)\,.
    \end{align*}
    This can be proven directly from the definition~\eqref{eq:defn_w2}.
    \item \textbf{A lower bound on the Wasserstein distance}.\label{fact:low_bdd_gaussian}
    
    Let $\mu, \nu \in \mc P_2(\R^d)$. If $\tilde \mu$ and $\tilde \nu$ are \emph{Gaussian} measures with the same moments up to order two as $\mu$ and $\nu$, respectively, then $W_2(\mu,\nu) \ge W_2(\tilde\mu,\tilde \nu)$~\cite{cuestaalbertosmatranbeatuerodiaz1996lowerbdw2}.
\end{enumerate}

\section{Proofs for the geodesic convexity results}\label{scn:proof_for_geod_cvx}

\subsection{Proof of Theorem~\ref{thm:sqrt_lambda_min_concave}}

See Appendix~\ref{scn:bures_geometry} and~\ref{scn:gen_geod_cvxty} for background on the relevant geometric concepts.

We begin by proving that the functionals $-\sqrt{\lambda_{\min}}$ and $\sqrt{\lambda_{\max}}$ are geodesically convex. We do this by establishing that these functionals are convex along barycenters, since that implies geodesic convexity (see Appendix~\ref{scn:gen_geod_cvxty}).
The following argument is implicit in the proofs of~\cite[Theorem 6.1]{agueh2011barycenter} and~\cite[Theorem 8]{bhatiajainlim2019bures}, and we include it for completeness.
    
\begin{thm}\label{thm:geod_cvxty_result}
    The functionals $-\sqrt{\lambda_{\min}} : \psd\to\R$ and $\sqrt{\lambda_{\max}} : \psd \to \R$ are convex along barycenters.
\end{thm}
\begin{proof}
    If $Q$ is a probability measure on $\psd$ with barycenter $\Sigstar$, then
    \begin{align*}
        \Sigstar
        &= \int {(\Sigma^{\star \, 1/2} \Sigma \Sigma^{\star \, 1/2})}^{1/2} \, \D Q(\Sigma)\,,
    \end{align*}
    see~\cite[Theorem 6.1]{agueh2011barycenter}. This implies
    \begin{align*}
        \lambda_{\min}(\Sigstar)
        &\ge \int \sqrt{\lambda_{\min}(\Sigma^{\star \, 1/2} \Sigma \Sigma^{\star \, 1/2})} \, \D Q(\Sigma)
        \ge \sqrt{\lambda_{\min}(\Sigstar)} \int \sqrt{\lambda_{\min}(\Sigma)} \, \D Q(\Sigma)\,,
    \end{align*}
    whence
    \begin{align*}
        \sqrt{\lambda_{\min}(\Sigstar)}
        &\ge \int \sqrt{\lambda_{\min}(\Sigma)} \, \D Q(\Sigma)\,.
    \end{align*}
    A similar argument applies for $\sqrt{\lambda_{\max}}$.
\end{proof}

\begin{rema}
    This result implies for instance that the set of PSD matrices with eigenvalues lying in a certain range is geodesically convex.
\end{rema}

Since the update for Bures--Wasserstein SGD only involves moving along geodesics, the above result already suffices to control the eigenvalues of the SGD iterates. However, the update for Bures--Wasserstein GD entails movement along \emph{generalized} geodesics, for which we need the control in Theorem~\ref{thm:sqrt_lambda_min_concave}.

Before proving Theorem~\ref{thm:sqrt_lambda_min_concave}, however, we provide some intuition for the proof. Denote by $F$ the barycenter functional $F(\Sigma) \deq \frac{1}{2} \int W_2^2(\Sigma,\cdot) \, \D Q$ corresponding to the measure $Q$ and for the sake of intuition, pretend that $\sqrt{\lambda_{\min}}$ is differentiable everywhere. Let ${(\Sigma_t)}_{t\ge 0}$ denote the gradient flow of $F$, i.e., $\dot \Sigma_t = -\nabla F(\Sigma_t)$. We observe that the gradient of $F$ can be written as an \emph{average}, hence
\begin{align*}
    \partial_t \sqrt{\lambda_{\min}(\Sigma_t)}
    &= -\langle \nabla \sqrt{\lambda_{\min}}(\Sigma_t), \nabla F(\Sigma_t)\rangle_{\Sigma_t}
    = \int \langle \nabla \sqrt{\lambda_{\min}}(\Sigma_t), \log_{\Sigma_t} \Sigma' \rangle_{\Sigma_t} \, \D Q(\Sigma')\,,
\end{align*}
see Fact~\ref{fact:grad_sq_dist} in Appendix~\ref{scn:facts}.
However, the geodesic concavity of $\sqrt{\lambda_{\min}}$ implies that
\begin{align*}
    \sqrt{\lambda_{\min}(\Sigma')}
    &\le \sqrt{\lambda_{\min}(\Sigma_t)} + \langle \nabla \sqrt{\lambda_{\min}}(\Sigma_t), \log_{\Sigma_t} \Sigma'\rangle_{\Sigma_t}
\end{align*}
and therefore
\begin{align*}
    \partial_t \sqrt{\lambda_{\min}(\Sigma_t)}
    \ge {\underbrace{\int\sqrt{\lambda_{\min}(\Sigma')} \, \D Q(\Sigma')}_{\eqqcolon \sqrt\alpha}} - \sqrt{\lambda_{\min}(\Sigma_t)}\,.
\end{align*}
This shows that as soon as $\lambda_{\min}(\Sigma_t)$ hits $\alpha$, then $\sqrt{\lambda_{\min}(\Sigma_t)}$ is increasing. Thus, the continuous-time gradient flow for $F$ always has eigenvalues at least $\alpha$ provided that it is initialized appropriately
and $\sqrt{\lambda_{\min}}$
is differentiable throughout its trajectory.

To summarize, the geodesic concavity of $\sqrt{\lambda_{\min}}$, together with the expression for the gradient of $F$ as an average of tangent vectors pointing towards matrices in the support of $Q$, yields eigenvalue control for the continuous-time gradient flow of $F$. This argument does not apply directly to the discrete-time GD updates, but nevertheless we show that the eigenvalues of the GD iterates can be controlled provided that the step size is taken sufficiently small; this is the content of Theorem~\ref{thm:sqrt_lambda_min_concave}.

\begin{proof}[Proof of Theorem~\ref{thm:sqrt_lambda_min_concave}]
    For $0 \le \eta \le 1$, let $\Sigma_\eta$ denote the generalized barycenter of the distribution $Q_\eta \deq (1-\eta) \, \delta_{\Sigma_0} + \eta \, Q$. If $\bar T$ denotes the average transport map
    \begin{align*}
        \bar T
        &\deq \int T_{\Sigma_0 \to \Sigma} \, \D Q(\Sigma)\,,
    \end{align*}
    then we have
    \begin{align*}
        \Sigma_\eta
        &= \bigl((1-\eta)\,I_d + \eta \, \bar T\bigr)\,\Sigma_0 \,\bigl((1-\eta)\,I_d + \eta \, \bar T\bigr)
        = {(1-\eta)^2} \,\Sigma_0 + \eta^2 \,\bar T \Sigma_0 \bar T + \eta \, (1-\eta) \, (\bar T \Sigma_0 + \Sigma_0 \bar T)\,.
    \end{align*}
    On the other hand, let $\gamma_\Sigma(\eta)$ denote the geodesic joining $\Sigma_0$ to $\Sigma$ at time $\eta$. Then,
    \begin{align*}
        \gamma_\Sigma(\eta)
        &= \bigl((1-\eta)\,I_d + \eta \, T_{\Sigma_0\to\Sigma}\bigr) \, \Sigma_0\,\bigl((1-\eta)\,I_d + \eta \, T_{\Sigma_0\to\Sigma}\bigr) \\
        &= {(1-\eta)}^2 \, \Sigma_0 + \eta^2 \, \Sigma + \eta \, (1-\eta) \, (T_{\Sigma_0\to\Sigma} \Sigma_0 + \Sigma_0 T_{\Sigma_0\to\Sigma})\,.
    \end{align*}
    Upon integrating w.r.t.\ $\D Q(\Sigma)$ and comparing the two expressions, we find that
    \begin{align*}
        \Sigma_\eta
        &= \int \gamma_\Sigma(\eta) \, \D Q(\Sigma) + \eta^2 \, \Bigl(\bar T\Sigma_0 \bar T - \int \Sigma \, \D Q(\Sigma)\Bigr)
        \succeq \int \gamma_\Sigma(\eta) \, \D Q(\Sigma) - \beta \eta^2 \, I_d\,.
    \end{align*}
    
    Next, using the geodesic concavity of $\sqrt{\lambda_{\min}}$ and Jensen's inequality,
    \begin{align*}
        \lambda_{\min}\Bigl(\int \gamma_\Sigma(\eta) \, \D Q(\Sigma)\Bigr)
        &\ge \int \lambda_{\min}\bigl(\gamma_\Sigma(\eta)\bigr) \, \D Q(\Sigma)
        \ge \int \bigl((1-\eta)\,\sqrt{\lambda_{\min}(\Sigma_0)} + \eta \,\sqrt{\lambda_{\min}(\Sigma)}\bigr)^2 \, \D Q(\Sigma) \\
        &\ge \Bigl(\int \bigl((1-\eta)\,\sqrt{\lambda_{\min}(\Sigma_0)} + \eta \,\sqrt{\lambda_{\min}(\Sigma)}\bigr) \, \D Q(\Sigma)\Bigr)^2 \\
        &\ge \bigl((1-\eta)\,\sqrt{\lambda_{\min}(\Sigma_0)} + \eta \,\sqrt{\alpha}\bigr)^2\,.
    \end{align*}
    We have established the inequality
    \begin{align*}
        \lambda_{\min}(\Sigma_\eta)
        &\ge \bigl((1-\eta)\,\sqrt{\lambda_{\min}(\Sigma_0)} + \eta \,\sqrt{\alpha}\bigr)^2 - \beta \eta^2\,.
    \end{align*}
    
    We now search for a value of $\lambda \ge 0$ such that if $\lambda_{\min}(\Sigma_0) \ge \lambda$, then $\lambda_{\min}(\Sigma_\eta) \ge \lambda$.
    From the above inequality, it suffices to have
    \begin{align*}
        \bigl((1-\eta)\,\sqrt{\lambda} + \eta \,\sqrt{\alpha}\bigr)^2 - \beta \eta^2 \overset{!}{\ge} \lambda\,.
    \end{align*}
    Rearranging this expression, we want
    \begin{align*}
        (\sqrt\alpha - \sqrt\lambda) \, \eta
        &\overset{!}{\ge} \sqrt{\lambda + \beta \eta^2} - \sqrt\lambda = \sqrt \lambda \, \Bigl(\sqrt{1+\frac{\beta \eta^2}{\lambda}} - 1\Bigr)\,.
    \end{align*}
    Applying the inequality $\sqrt{1+x} \le 1+x/2$, valid for $x\ge 0$, it suffices to have
    \begin{align*}
        (\sqrt \alpha - \sqrt\lambda) \, \eta
        \overset{!}{\ge} \frac{\beta \eta^2}{2\sqrt\lambda}\,.
    \end{align*}
    We now choose $\lambda = \alpha/4$, for which it can be verified that the above inequality holds for $\eta \le \frac{\alpha}{2\beta}$. This concludes the proof.
\end{proof}

\begin{rema}\label{rem:what_went_wrong}
    In an earlier version of this paper, we claimed that $-\sqrt{\lambda_{\min}}$ and $\sqrt{\lambda_{\max}}$ are convex along generalized geodesics, which is stronger than the statement of Theorem~\ref{thm:sqrt_lambda_min_concave}.
    Unfortunately, our proof of this claim was incorrect, as it relied upon~\cite[Corollary 3.5]{lawsonlim2001geometricmean} which is false as written.\footnote{The ``if'' direction of the corollary is incorrect: upon taking $B = I_d$, it says that $X \preceq B^{1/2}$ implies $X^2 \preceq B$, which contradicts the well-known fact that the square function is not operator monotone.} In fact, we have discovered a counterexample to our original claim: set
    \begin{align*}
        \Sigma_0 \deq \begin{bmatrix} 0.16 & 0.2 \\ 0.2 & 0.82 \end{bmatrix}\,, \qquad \Delta \deq \begin{bmatrix} 0.8 & 0.4 \\ 0.4 & 0.2 \end{bmatrix}\,, \qquad \Sigma \deq I_2 + \Sigma_0^{-1/2} \Delta \Sigma_0^{-1/2}\,.
    \end{align*}
    Let $Q \deq \frac{1}{2} \,\delta_{I_2} + \frac{1}{2} \, \delta_\Sigma$ and note that $\Sigma \succeq I_2$, i.e., $Q$ is supported on matrices with eigenvalues at least $1$.
    We can compute
    \begin{align*}
        T_{\Sigma_0 \to I_2}
        &= \Sigma_0^{-1/2}\,, \\
        T_{\Sigma_0\to \Sigma}
        &= \Sigma_0^{-1/2}\, \bigl(\Sigma_0^{1/2}\, (I_d + \Sigma_0^{-1/2} \Delta \Sigma_0^{-1/2})\,\Sigma_0^{1/2}\bigr)^{1/2}\, \Sigma_0^{-1/2}
        = \Sigma_0^{-1/2}\, (\Sigma_0+\Delta)^{1/2}\, \Sigma_0^{-1/2}\,, \\
        \bar T
        &= \frac{1}{2}\, \Sigma_0^{-1/2}\, \bigl(\Sigma_0^{1/2} + (\Sigma_0+\Delta)^{1/2}\bigr)\,\Sigma_0^{-1/2}\,,
    \end{align*}
    so that the generalized barycenter $\bar\Sigma$ of $Q$ at $\Sigma_0$ is
    \begin{align*}
        \bar\Sigma
        &=\Sigma_0^{-1/2}\,\Bigl(\frac{\Sigma_0^{1/2} + (\Sigma_0+\Delta)^{1/2}}{2}\Bigr)^2 \, \Sigma_0^{-1/2}\,.
    \end{align*}
    However, it can be numerically verified that $\lambda_{\min}(\bar\Sigma) \le 0.993 < 1$. This shows that the set of PSD matrices with eigenvalues at least $1$ is not closed under generalized geodesics.
    In particular, $-\sqrt{\lambda_{\min}}$ is an example of a functional which is convex along barycenters but not along generalized geodesics, which may be of interest in its own right. The revised statement of Theorem~\ref{thm:sqrt_lambda_min_concave} fixes this issue, at the cost of slightly worsening our quantitative results.
    
    We also remark that the above counterexample was obtained as follows. One can show that the statement
    \begin{center}
        the set of PSD matrices with eigenvalues at least $1$ is closed under generalized geodesics
    \end{center}
    is equivalent to the statement
    \begin{center}
        for all $\Sigma_0 \succ 0$ and all $A, B \succeq \Sigma_0$, it holds that
        \begin{align*}
            \Bigl( \frac{A^{1/2} + B^{1/2}}{2}\Bigr)^2 \succeq \Sigma_0\,.
        \end{align*}
    \end{center}
    The equivalence between the two statements is obtained by considering the generalized barycenter of the distribution $P \deq \frac{1}{2} \,\delta_{\Sigma_0^{-1/2} A \Sigma_0^{-1/2}} + \frac{1}{2} \,\delta_{\Sigma_0^{-1/2} B \Sigma_0^{-1/2}}$ at $\Sigma_0$. Therefore, we discovered our counterexample by finding a counterexample to the latter statement. Note also the similarity of the second statement with the last conjecture in~\cite{chakwo1985matrixineq}.
    In contrast, it was shown in~\cite[Lemma 10]{chewietal2020buresgd} that the set of matrices with eigenvalues \emph{at most} $\beta$ is convex along generalized geodesics.
\end{rema}

\subsection{Sharpness of Theorem~\ref{thm:geod_cvxty_result}}

We investigate the sharpness of this result in the following sense: for what exponents $p\in \R$ is it true that the functionals $-\lambda_{\min}^p$, $\lambda_{\max}^p$ are geodesically convex? For instance, the functional $\lambda_{\max}$ was shown to be geodesically convex in~\cite[Lemma 13]{chewietal2020buresgd}. 

In the following theorem, we show that the exponent $p=1/2$ in Theorem~\ref{thm:geod_cvxty_result} is optimal, in the sense that all possible geodesic convexity statements involving powers of $\lambda_{\min}$ and $\lambda_{\max}$ (except the trivial case $p=0$) can be deduced from the result for $p=1/2$.

\begin{thm}
    The following diagrams depict the exponents $p\in\R$ for which $\lambda_{\min}^p$ and $\lambda_{\max}^p$ are concave or convex.
    \begin{center}
    \begin{tikzpicture}
  \draw[<->] (3, 0) -- (-2, 0) node[left] {$p$};
  \node at (-2, 1) [left] {convex};
  \node at (-2, 2) [left] {concave};
  \node at (-2, 3) [left] {$\lambda_{\min}^p$};
  \node at (0, 0) [below] {$0$};
  \node at (1, 0) [below] {$\frac{1}{2}$};
  \draw[thick] (0, 2) -- (1, 2);
  \draw[->, thick] (0, 1) -- (-1.5, 1);
  \filldraw [black] (0,2) circle (2pt);
  \filldraw [black] (0,1) circle (2pt);
  \filldraw [black] (1,2) circle (2pt);
\end{tikzpicture}
\hfill
    \begin{tikzpicture}
  \draw[<->] (-2, 0) -- (3, 0) node[right] {$p$};
  \node at (3, 1) [right] {convex};
  \node at (3, 2) [right] {concave};
  \node at (3, 3) [right] {$\lambda_{\max}^p$};
  \node at (0, 0) [below] {$0$};
  \node at (1, 0) [below] {$\frac{1}{2}$};
  \draw[->, thick] (1, 1) -- (2.5, 1);
  \filldraw [black] (1,1) circle (2pt);
  \filldraw [black] (0,1) circle (2pt);
  \filldraw [black] (0,2) circle (2pt);
\end{tikzpicture}
\end{center}
The diagram is to be interpreted as follows. If part of the diagram is filled in with a solid black line, then the corresponding functional is geodesically concave/convex. If part of the diagram is \emph{not} filled in, then there exist counterexamples showing that the functional is \emph{not} geodesically concave/convex.
\end{thm}
\begin{proof}
    First, we establish the positive results, which follow from composition rules:
    \begin{itemize}
        \item For $0 \le p \le 1/2$, $\lambda_{\min}^p$ is the composition of the increasing concave function ${(\cdot)}^{2p}$ with the concave function $\sqrt{\lambda_{\min}}$, so it is concave.
        \item For $p\le 0$, $\lambda_{\min}^p$ is the composition of the decreasing convex function ${(\cdot)}^{2p}$ with the concave function $\sqrt{\lambda_{\min}}$, so it is convex.
        \item For $p\ge 1/2$, $\lambda_{\max}^p$ is the composition of the increasing convex function ${(\cdot)}^{2p}$ with the convex function $\sqrt{\lambda_{\max}}$, so it is convex.
    \end{itemize}
    Next, we turn towards the negative results. First, recall from Fact~\ref{fact:diagonal} in Appendix~\ref{scn:facts} that if $\Sigma_0$ and $\Sigma_1$ are one-dimensional, i.e., they are positive numbers, then the Bures--Wasserstein geodesic is
    \begin{align*}
        \Sigma_t
        &= {\bigl((1-t) \Sigma_0^{1/2} + t \Sigma_1^{1/2}\bigr)}^2\,, \qquad t\in [0,1]\,.
    \end{align*}
    Also, in this case, $\lambda_{\min}$ and $\lambda_{\max}$ coincide and equal the identity; we thus abuse notation slightly in this paragraph by writing $\lambda$ for both to handle the two cases simultaneously. Once we reparametrize by the square roots, it is seen that asking for concavity/convexity of $\lambda^p$ is equivalent to asking for usual convexity of ${(\cdot)}^{2p}$ on $\R_+$. This example rules out: (1) the concavity of $\lambda^p$ for $p < 0$; (2) the convexity of $\lambda^p$ for $0 < p < 1/2$; and (3) the concavity of $\lambda^p$ for $p > 1/2$.
    
    To rule out convexity of $\lambda_{\min}^p$ for $p > 0$, consider $\Sigma = \on{diag}(\varepsilon, 1/\varepsilon)$ for small $\varepsilon > 0$. The transport map from $\Sigma^{-1}$ to $\Sigma$ is $\Sigma$, so from~\eqref{eq:bures_geod} the midpoint of this geodesic is $M  \deq  (\Sigma + \Sigma^{-1} + 2I_2)/4 = (\varepsilon + \varepsilon^{-1} + 2) I_2 / 4$. In particular, this implies that $\lambda_{\min}(M) \geq 1/(4\varepsilon) \gg \varepsilon = \max\{\lambda_{\min}(\Sigma), \lambda_{\min}(\Sigma^{-1})\}$. Thus $\lambda_{\min}^p$ is not convex for any $p > 0$.

    To rule out concavity of $\lambda_{\max}^p$ for $p > 0$, note that for $\varepsilon$ sufficiently small, in the previous example $\lambda_{\max}(M) \approx 1/(4\varepsilon) \ll 1/\varepsilon = \max\{\lambda_{\max}(\Sigma), \lambda_{\max}(\Sigma^{-1})\}$. Also, for any $p < 0$, the convexity of $\lambda_{\max}^p$ would imply the concavity of $\lambda_{\max}^{-p}$ due to the composition rules, hence $\lambda_{\max}^p$ is not convex.
    
    This covers all cases.
\end{proof}

\subsection{Eigenvalue clipping is a Bures--Wasserstein contraction}

Convex sets play an important role in Euclidean optimization because projection onto a convex set is a contraction (c.f.~\cite[Lemma 3.1]{bubeck2015convex}), and hence projected gradient descent can be used to solve constrained optimization. Unfortunately, as the Bures--Wasserstein space is positively curved, we cannot automatically conclude that projection onto a geodesically convex set is a projection. Nevertheless, we can verify by hand the following result. In what follows, define for $0 < \beta < \infty$ the operator $\clip^\beta : \psd \to \psd$ in the following way: if $\Sigma = \sum_{i=1}^d \lambda_i u_i u_i^\T$ is an eigenvalue decomposition of $\Sigma$, then \[ \clip^\beta \Sigma  \deq  \sum_{i=1}^d (\lambda_i \wedge \beta) \, u_i u_i^\T\,.\]

\begin{prop}\label{thm:eig_clip_bures_contraction}
    The operator $\clip^\beta$ is a contraction w.r.t.\ the Bures--Wasserstein metric, i.e., $W_2(\clip^\beta \Sigma, \clip^\beta \Sigma') \le W_2(\Sigma, \Sigma')$.
\end{prop}

To prove this proposition, we first extend the clipping operation to an operator $\R^{d\times d} \to \R^{d\times d}$ via the singular values; namely, given a singular value decomposition $A = \sum_{i=1}^d s_i u_i v_i^\T$, we let $\clip^\beta A  \deq  \sum_{i=1}^d (s_i \wedge \beta) \, u_i v_i^\T$.

\begin{proof}[Proof of Proposition~\ref{thm:eig_clip_bures_contraction}]
    Fix $X, Y \in \psd$.
    It is known (see, e.g.,~\cite{bhatiajainlim2019bures}) that
    \begin{align*}
        W_2(X, Y)
        &= \min_{\substack{A, B \in \R^{d\times d} \\ AA^\T = X \\ BB^\T = Y}}{\norm{A-B}_{\rm F}}\,.
    \end{align*}
    Let $(\bar A, \bar B)$ be a minimizing pair in the above expression.
    We aim to show
    \begin{align*}
        W_2(\clip^\beta X, \clip^\beta Y)
        &\le \norm{\clip^{\sqrt \beta} \bar A - \clip^{\sqrt \beta} \bar B}_{\rm F}
        \overset{?}{\le} \norm{\bar A - \bar B}_{\rm F}
        = W_2(X, Y)\,.
    \end{align*}
    We only have to show the second inequality, and we do so by showing that the operator $\clip^M : \R^{d\times d} \to \R^{d\times d}$ satisfies
    \begin{align}\label{eq:clip_proj}
        \clip^M A
        &= \argmin_{\tilde A \in \R^{d\times d}, \; \norm{\tilde A} \le M}{\norm{A-\tilde A}_{\rm F}}\,, \qquad A \in \R^{d\times d}\,.
    \end{align}
    This will prove that $\clip^M$ is the Euclidean \emph{projection} onto the closed convex set $\{\norm \cdot \le M\}$, and such a projection is automatically $1$-Lipschitz.

    Indeed, showing~\eqref{eq:clip_proj} is standard.
    Write $A = U\Sigma V^\T$ for its singular value decomposition.
    \begin{align*}
        &\argmin_{\tilde A\in\R^{d\times d}, \; \norm{\tilde A} \le M}{\norm{\tilde A - A}_{\rm F}^2}
        = \argmin_{\tilde A \in \R^{d\times d}, \; \norm{\tilde A} \le M}{\norm{\tilde A - U\Sigma V^\T}_{\rm F}^2}
        = \argmin_{\tilde A \in \R^{d\times d}, \; \norm{\tilde A} \le M}{\norm{U^\T \tilde AV - \Sigma}_{\rm F}^2} \\
        &\qquad = \argmin_{\tilde A \in \R^{d\times d}, \; \norm{\tilde A} \le M} \Bigl\{\sum_{i=1}^d {\{\Sigma[i,i] - (U^\T \tilde A V)[i,i]\}}^2 + \sum_{\substack{i,j\in [d] \\ i \ne j}} {(U^\T \tilde A V)[i,j]}^2 \Bigr\}\,.
    \end{align*}
    On the other hand,
    \begin{align*}
        &\min_{\tilde A \in \R^{d\times d}, \; \norm{\tilde A} \le M} \Bigl\{\sum_{i=1}^d {\{\Sigma[i,i] - (U^\T \tilde A V)[i,i]\}}^2 + \sum_{i,j\in [d], \; i \ne j} {(U^\T \tilde A V)[i,j]}^2 \Bigr\} \\
        &\qquad \ge \sum_{i=1}^d {\{{(\Sigma[i,i] - M)}_+\}}^2\,,
    \end{align*}
    with equality attained at the unique minimizer $\tilde A$ satisfying $U^\T \tilde A V = \clip^M \Sigma$, i.e., $\tilde A = \clip^M A$.
\end{proof}

\section{Proofs for barycenters}

\subsection{Riemannian gradient descent}\label{scn:bures_gd}

In this section, we detail the obstacles faced by previous analyses
and then show how our geometric result,
Theorem~\ref{thm:sqrt_lambda_min_concave}, enables us to overcome the
prior exponential dependence on dimension and obtain the dimension-free rates in Theorems~\ref{thm:bures_gd} and~\ref{thm:bures_avg_case}.

We begin by recalling the proof strategy of~\cite{chewietal2020buresgd}. Let $F$ denote the barycenter functional,
\begin{align}\label{eq:bary_functional}
    F(\Sigma)
    & \deq  \frac{1}{2} \int W_2^2(\Sigma,\cdot) \, \D P\,.
\end{align}
Standard optimization guarantees are often proven under the assumption that the objective function $F$ is smooth and convex. Since we are considering Riemannian descent, this should be interpreted as convex and smooth along geodesics, as in~\cite{zhangsra2016geodesicallycvx}. Unfortunately, the functional $F$ is not geodesically convex (see~\cite[Appendix B.2]{chewietal2020buresgd}), and so we must look for weaker conditions which still imply convergence of GD/SGD. A gradient domination condition known as the \emph{Polyak--\L{}ojasiewicz inequality} (henceforth \emph{PL inequality}) was introduced in the non-convex optimization literature as an appropriate substitute for strong convexity~\cite{kariminutinischmidt2016pl}, and it plays a key role in the analysis.

The following properties of the barycenter functional were proven in~\cite{chewietal2020buresgd}.

\begin{thm}\label{thm:properties_of_bary_fn}
    Let $0 < \lambda_{\min} \le \lambda_{\max} <\infty$ and write $\kappa  \deq  \lambda_{\max}/\lambda_{\min}$.
    \begin{enumerate}
        \item (\cite[Theorem 7]{chewietal2020buresgd}) The barycenter functional $F$ is $1$-geodesically smooth.
        \item (\cite[Theorem 17]{chewietal2020buresgd}) Assume that the covariance matrices in the support of $P$ have eigenvalues in the range $[\lambda_{\min},\lambda_{\max}]$. Then, $F$ satisfies a \emph{variance inequality},
        \begin{align*}
            F(\Sigma) - F(\Sigstar)
            &\ge \frac{1}{2\kappa} \, W_2^2(\Sigma, \Sigstar)\,, \qquad\text{for all}~\Sigma \in \psd\,.
        \end{align*}
        \item (\cite[Theorem 19]{chewietal2020buresgd}) Assume that the covariance matrices in the support of $P$, as well as $\Sigma$ itself, have eigenvalues in the range $[\lambda_{\min},\lambda_{\max}]$.
        Then, $F$ satisfies a PL inequality at the matrix $\Sigma$:
        \begin{align*}
            F(\Sigma) - F(\Sigstar)
            &\le 2\kappa^2 \, \norm{\nabla F(\Sigma)}_\Sigma^2\,.
        \end{align*}
    \end{enumerate}
\end{thm}

Geodesic smoothness together with a PL inequality at every iterate are enough to obtain convergence guarantees for GD/SGD in objective value (i.e., the quantity $F(\Sigma) - F(\Sigstar)$), c.f.~\cite[Theorems 4-5]{chewietal2020buresgd}. The variance inequality is then used to deduce convergence of the iterate to $\Sigstar$.

The main difficulty when applying these results is the assumption required for the third point: it requires \emph{a priori} control over the eigenvalues of the iterates of GD/SGD.

This difficulty is addressed in~\cite{chewietal2020buresgd} via the following strategy: identify a geodesically convex subset $\mc S$ of the Bures--Wasserstein manifold for which we can prove uniform bounds on the eigenvalues of matrices in $\mc S$. Since the iterates of SGD travel along geodesics, if $P$ is supported in $\mc S$ and the algorithm is initialized in $\mc S$, it follows that all iterates of SGD will remain in $\mc S$. The situation is similar for GD, except that ``geodesics'' must be replaced by ``generalized geodesics''.

We can now describe the source of the exponential dependence on dimension in the result of~\cite{chewietal2020buresgd}: if the covariance matrices in the support of $P$ have eigenvalues in $[\lambda_{\min},\lambda_{\max}]$, then the subset $\mc S$ used in the analysis of Chewi et al.\ is substantially larger than the support of $P$, and in particular the eigenvalues of matrices in $\mc S$ can only be proven to lie in the range $[\lambda_{\min}/\kappa^{d-1}, \lambda_{\max}]$. The main improvement in the present analysis is to use our geometric result (Theorem~\ref{thm:sqrt_lambda_min_concave}) to prove the following result.

\begin{lemma}\label{lem:eigvals_of_iterates}
    Suppose that the covariance matrices in the support of $P$ have eigenvalues in the range $[\lambda_{\min},\lambda_{\max}]$, and that we initialize GD (respectively SGD) at a point in $\supp P$. Then, the iterates of GD with step size at most $\frac{1}{2\kappa}$ (respectively SGD) also have eigenvalues in the range $[\lambda_{\min}/4,\lambda_{\max}]$ (respectively $[\lambda_{\min}, \lambda_{\max}]$).
\end{lemma}
\begin{proof}
    The result for SGD follows because SGD moves along geodesics and the set of matrices with eigenvalues in $[\lambda_{\min},\lambda_{\max}]$ is geodesically convex (Theorem~\ref{thm:geod_cvxty_result}). For GD, we instead invoke the generalized geodesic convexity of $\lambda_{\max}$ (see~\cite[Lemma 10]{chewietal2020buresgd}) together with Theorem~\ref{thm:sqrt_lambda_min_concave}.
\end{proof}

This combined with the arguments below is enough to alleviate the exponential dimension dependence. However, before continuing to the main argument, we prove sharper bounds for the last two statements of Theorem~\ref{thm:properties_of_bary_fn}. This allows us to also improve our convergence rates' dependence on the conditioning. 

This improved version of Theorem~\ref{thm:properties_of_bary_fn} rests on the following observation.~\cite[Lemma 16]{chewietal2020buresgd} shows that if $\Sigma$, $\Sigma'$ have eigenvalues which lie in the range $[\lambda_{\min}, \lambda_{\max}]$, then the eigenvalues of the transport map $T_{\Sigma\to\Sigma'}$ lie in $[\kappa^{-1}, \kappa]$. However, these bounds are loose, as following lemma shows.

\begin{lemma}\label{lem:transport_map_reg}
    Suppose that $\Sigma,\Sigma' \in \psd$ have eigenvalues which lie in the range $[\lambda_{\min}, \lambda_{\max}]$, and let $\kappa  \deq  \lambda_{\max}/\lambda_{\min}$ denote the condition number.
    Then, the eigenvalues of the transport map $T_{\Sigma\to\Sigma'}$ lie in the range $[1/\sqrt \kappa, \sqrt \kappa]$.
\end{lemma}
\begin{proof}
    The transport map $T_{\Sigma\to\Sigma'}$ is explicitly given in~\eqref{eq:gaussian_ot_map}, and it can be recognized as the matrix geometric mean of $\Sigma^{-1}$ and $\Sigma'$. Applying a norm bound for the matrix geometric mean~\cite[Theorem 3]{bhatiagrover2012normineq}, we deduce that
    \begin{align*}
        \lambda_{\max}(T_{\Sigma\to\Sigma'})
        &\le \lambda_{\max}(\Sigma^{' \, 1/4} \Sigma^{-1/2} \Sigma^{' \, 1/4})
        \le \sqrt{\kappa}\,.
    \end{align*}
    The symmetry of $\Sigma$ and $\Sigma'$ together with Fact~\ref{fact:inv_transport_map} in Appendex~\ref{scn:facts} yields $\lambda_{\min}(T_{\Sigma\to\Sigma'}) \ge 1/\sqrt\kappa$.
\end{proof}

Using this lemma, we now state and prove the refinement of Theorem~\ref{thm:properties_of_bary_fn}.

\begin{thm}\label{thm:better_properties_of_bary_fn}
    Let $0 < \lambda_{\min} \le \lambda_{\max} <\infty$ and write $\kappa  \deq  \lambda_{\max}/\lambda_{\min}$.
    \begin{enumerate}
        \item (\cite[Theorem 7]{chewietal2020buresgd}) The barycenter functional $F$ is $1$-geodesically smooth.
        \item Assume that the covariance matrices in the support of $P$ have eigenvalues in the range $[\lambda_{\min},\lambda_{\max}]$. Then, $F$ satisfies a \emph{variance inequality},
        \begin{align*}
            F(\Sigma) - F(\Sigstar)
            &\ge \frac{1}{2\sqrt \kappa} \, W_2^2(\Sigma, \Sigstar)\,, \qquad\text{for all}~\Sigma \in \psd\,.
        \end{align*}
        \item Assume that the covariance matrices in the support of $P$ have eigenvalues in the range $[\lambda_{\min},\lambda_{\max}]$.
        Then, $F$ satisfies a PL inequality at the matrix $\Sigma$:
        \begin{align*}
            F(\Sigma) - F(\Sigstar)
            &\le 2\sqrt\kappa\,\frac{\lambda_{\max}}{\lambda_{\min}(\Sigma)} \, \norm{\nabla F(\Sigma)}_\Sigma^2\,.
        \end{align*}
    \end{enumerate}
\end{thm}
\begin{proof}
    The second statement follows from the general variance inequality (\cite[Theorem 6]{chewietal2020buresgd}) together with Lemma~\ref{lem:transport_map_reg}. Similarly, the third statement follows from the proof of~\cite[Theorem 19]{chewietal2020buresgd} using the improved variance inequality.
\end{proof}

We can now prove Theorem~\ref{thm:bures_gd}.

\begin{proof}[Proof of Theorem~\ref{thm:bures_gd}]
    The proof for SGD follows from~\cite[Theorem 5]{chewietal2020buresgd}. For GD, we first note that from Lemma~\ref{lem:eigvals_of_iterates} and Theorem~\ref{thm:better_properties_of_bary_fn}, we have the PL inequality
    \begin{align*}
        F(\Sigma_t^{\rm GD}) - F(\Sigstar)
        &\le 8\kappa^{3/2} \, \norm{\nabla F(\Sigma_t^{\rm GD})}_{\Sigma_t^{\rm GD}}^2
    \end{align*}
    at any GD iterate $\Sigma_t^{\rm GD}$.
    Also, from the $1$-smoothness of the barycenter functional, we obtain the descent lemma
    \begin{align*}
        F(\Sigma_{t+1}^{\rm GD}) - F(\Sigma_t^{\rm GD})
        &\le -\eta \, \bigl(1 -\frac{\eta}{2}\bigr) \, \norm{\nabla F(\Sigma_t^{\rm GD})}_{\Sigma_t^{\rm GD}}^2\,.
    \end{align*}
    With our step size choice $\eta = \frac{1}{2\kappa}$, this becomes
    \begin{align*}
        F(\Sigma_{t+1}^{\rm GD}) - F(\Sigma_t^{\rm GD})
        &\le -\frac{3}{8\kappa}\, \norm{\nabla F(\Sigma_t^{\rm GD})}_{\Sigma_t^{\rm GD}}^2\,.
    \end{align*}
    Combining these two inequalities and iterating yields the result for GD\@.
\end{proof}

We now sketch the modifications required to prove Theorem~\ref{thm:bures_avg_case}.

\begin{proof}[Proof of Theorem~\ref{thm:bures_avg_case}]
It will be convenient
to define
$$
 \norm{\lambda_{\max}}_{1/2}
        \deq \Bigl(\int \sqrt{\lambda_{\max}(\Sigma)} \, \D P(\Sigma)\Bigr)^2\,.
$$
    
    We begin by checking that the variance inequality and PL inequality from Theorem~\ref{thm:better_properties_of_bary_fn} continue to hold under these assumptions.
    
    \textbf{Variance inequality}. From the geodesic convexity of $-\sqrt{\lambda_{\min}}$ and $\sqrt{\lambda_{\max}}$, the barycenter $\Sigstar$ of $P$ has eigenvalues in $[\norm{\lambda_{\min}}_{1/2}, \norm{\lambda_{\max}}_{1/2}]$. By modifying the proof of Lemma~\ref{lem:transport_map_reg} and using Fact~\ref{fact:inv_transport_map} in Appendix~\ref{scn:facts}, the transport map $T_{\Sigstar\to\Sigma}$ has eigenvalues bounded below as
    \begin{align*}
        \lambda_{\min}(T_{\Sigstar\to\Sigma})
        &= \frac{1}{\lambda_{\max}(T_{\Sigma\to\Sigstar})}
        \ge \frac{1}{\lambda_{\max}(\Sigma^{\star \, 1/4} \Sigma^{-1/2} \Sigma^{\star \, 1/4})}
        \ge \frac{{\lambda_{\min}(\Sigma)}^{1/2}}{\norm{\lambda_{\max}}_{1/2}^{1/2}}\,.
    \end{align*}
    From~\cite[Theorem 6]{chewietal2020buresgd}, we can deduce that the variance inequality holds for $P$ with constant
    \begin{align*}
        \int \lambda_{\min}(T_{\Sigstar\to\Sigma}) \, \D P(\Sigma)
        &\ge \Bigl(\frac{\norm{\lambda_{\min}}_{1/2}}{\norm{\lambda_{\max}}_{1/2}}\Bigr)^{1/2}\,.
    \end{align*}
    
    \textbf{PL inequality}. Similarly, a modification of the proof of~\cite[Theorem 19]{chewietal2020buresgd} using the improved variance inequality shows that a PL inequality holds at $\Sigma$:
    \begin{align*}
        F(\Sigma) - F(\Sigstar)
        &\le 2\,\Bigl(\frac{\norm{\lambda_{\max}}_{1/2}}{\norm{\lambda_{\min}}_{1/2}}\Bigr)^{1/2} \, \frac{\norm{\lambda_{\max}}_{1/2}}{\lambda_{\min}(\Sigma)} \, \norm{\nabla F(\Sigma)}_\Sigma^2\,.
    \end{align*}
    
    \textbf{Putting it together}. 
 By Theorem~\ref{thm:sqrt_lambda_min_concave}, the iterates of GD with step size at most $\frac{\norm{\lambda_{\min}}_{1/2}}{2\,\norm{\lambda_{\max}}_1}$ all have eigenvalues in the range $[\norm{\lambda_{\min}}_{1/2}/4, \norm{\lambda_{\max}}_1]$.
    Using $\norm{\lambda_{\max}}_{1/2} \le \norm{\lambda_{\max}}_1$, the proof is concluded as before.
\end{proof}

\subsection{Euclidean gradient descent approach}\label{scn:euclidean_gd}

We now present our results for the Euclidean geometry.~\cite{bhatiajainlim2019bures} prove that the barycenter functional is strictly convex on the positive semidefinite cone (w.r.t.\ the standard Euclidean geometry). We extend their results by showing that it is in fact strongly convex and smooth (again w.r.t.\ the standard Euclidean geometry). Besides yielding an analysis of Euclidean projected GD and SGD, these results also aid our analysis of the regularized barycenter problem in the sequel.

Fix $0 < \alpha \leq \beta$ and denote by $\mathcal{K}_{\alpha,\beta}$ the subset of covariance matrices whose spectrum lies within $[\alpha, \beta]$. Let $F$ denote the barycenter functional, defined in~\eqref{eq:bary_functional}.

\begin{lemma}\label{lem:strong convexity smoothness}
For all $\Sigma \in \mathcal{K}_{\alpha,\beta}$ and non-zero $Y \in \sym$,
\begin{equation}
    \frac{\alpha^3}{4\beta^4} \leq \frac{\langle Y, \operatorname{D}^2 F(\Sigma)[Y] \rangle_{\rm F}}{\norm{Y}_{\rm F}^2} \leq \frac{\beta^2}{4\alpha^{3}}\,. 
\end{equation}
\end{lemma}
\begin{proof}
It suffices to consider the case where $P = \frac 1N \sum_{i=1}^N \delta_{\Sigma_i}$ for some $\Sigma_i\in\mathcal{K}_{\alpha,\beta},\;i \in [N]$, as the case of general $P$ supported on $\mathcal{K}_{\alpha,\beta}$ follows by compactness. Fix $\Sigma \in \mathcal{K}_{\alpha,\beta}$. Standard calculations as in~\cite{bhatiajainlim2019bures} show that the first derivative satisfies
\begin{equation*}
    2\operatorname{D}F(\Sigma) = I_d - \frac 1N \sum\limits_{i=1}^N \GM(\Sigma_i, \Sigma^{-1})\,.
\end{equation*}
We now compute the second derivative. Define the functions
\begin{align*}
    \inv(\Sigma) & \deq  \Sigma^{-1} \,,\\
    \con_A(\Sigma) & \deq  A\Sigma A \,, \\
    \sq(\Sigma) & \deq  \Sigma^{1/2}\,.
\end{align*}
For $Y \in \sym$, the above maps have derivatives
\begin{align*}
    \operatorname{D}\inv(\Sigma)[Y] &= -\Sigma^{-1}Y\Sigma^{-1}\,, \\
    \operatorname{D}\con_A(\Sigma)[Y] &= AYA\,, \\
    \operatorname{D}\sq(\Sigma)[Y] &= \int_0^\infty e^{-t\Sigma^{1/2}} Y e^{-t \Sigma^{1/2}} \, \D t\,.
\end{align*}
With these definitions in hand, we can write
\begin{equation*}
    2\operatorname{D}F(\Sigma) = I_d - \frac 1N \sum\limits_{i=1}^N \con_{\Sigma_i^{1/2}} \circ \sq \circ \con_{\Sigma_i^{-1/2}} \circ \inv(\Sigma)\,. 
\end{equation*}
Taking the derivative in a symmetric direction $Y \in \sym$ and applying the chain rule repeatedly,
\begin{align*}
    \begin{aligned}
    &2\operatorname{D}^2 F(\Sigma)[Y] \\
    &\qquad= \frac 1N \sum\limits_{i=1}^N \int_0^\infty \Sigma_i^{1/2} e^{-t \, {(\Sigma_i^{1/2} \Sigma \Sigma_i^{1/2})}^{-1/2}} \Sigma_i^{-1/2} \Sigma^{-1} Y \Sigma^{-1} \Sigma_i^{-1/2} e^{-t \, {(\Sigma_i^{1/2} \Sigma \Sigma_i^{1/2})}^{-1/2}} \Sigma_i^{1/2} \,\D t\,. 
    \end{aligned}
\end{align*}
Let $g(t,x) = \exp(-t/\sqrt{x}) \,x^{-1}$ on $(t,x) \in (0,\infty)\times (0,\infty)$ and $Z_i = \Sigma_i^{1/2}\Sigma \Sigma_i^{1/2}$. Since $g(t,\cdot)$ is analytic on its domain, the Riesz--Dunford calculus (see~\cite{dunford1988linear}) applies and we may write
\begin{align*}
    2\,\langle Y, \operatorname{D}^2 F(\Sigma)[Y] \rangle_\text{F} &= \frac 1N \sum\limits_{i=1}^N \int_0^\infty \tr\bigl(g(t, Z_i) \Sigma_i^{1/2}Y\Sigma_i^{1/2}g(t,Z_i) \Sigma_i^{1/2}Y\Sigma_i^{1/2}\bigr)\,\D t\,. \\
\intertext{Using the spectral mapping theorem and Lemma~\ref{lem:trace inequality} below we further write}
&\ge \frac{\norm{Y}_\text{F}^2}{N} \sum\limits_{i=1}^N {\lambda_{\min}(\Sigma_i)}^2 \int_0^\infty \min\limits_{\lambda \in \operatorname{spec}(Z_i)} {{g(t, \lambda)}^2} \, \D t\,.
\end{align*}
To bound the integral, we note that
\begin{equation*}
    e^{-t/\sqrt{\lambda_{\min}(Z_i)}} \, {\lambda_{\max}(Z_i)}^{-1} \leq g(t, \lambda)
\end{equation*}
for all $\lambda \in \operatorname{spec}(Z_i)$. Since we assume $\alpha I_d \preceq \Sigma_i, \Sigma \preceq \beta I_d$, then $\alpha^2 I_d \preceq Z_i \preceq \beta^2 I_d$, so
\begin{align*}
    \frac{2\,\langle Y, \operatorname{D}^2 F(\Sigma)[Y] \rangle_\text{F}}{\norm Y_{\rm F}^2}
    \ge \alpha^2 \int_0^\infty \exp\bigl( - \frac{2t}{\alpha}\bigr) \, \frac{1}{\beta^4} \, \D t
    = \frac{\alpha^3}{2\beta^4}\,.
\end{align*}

For the upper bound, an analogous calculation gives
\begin{align*}
    \frac{\langle Y, \on{D}^2 F(\Sigma)[Y]\rangle_{\rm F}}{\norm Y_{\rm F}^2}
    &\le \frac{\beta^3}{4\alpha^4}\,.
\end{align*}
However, the upper bound can be sharpened to $\frac{\beta^2}{4\alpha^3}$, see~\cite[Theorem 3.1]{kumyun2019gradientproj}.
\end{proof}

\begin{rema}
Similar to Theorem~\ref{thm:bures_avg_case}, one can obtain improved strong convexity and smoothness parameters for $F$ based on non-uniform notions of conditioning.
\end{rema}

We can now describe the projected gradient descent and projected stochastic gradient updates.
Let $\Pi_{\alpha,\beta} : \sym \to \mathcal{K}_{\alpha,\beta}$ denote the Euclidean projection onto $\mc K_{\alpha,\beta}$ and let $\eta = 4\lambda_{\min}^3/\lambda_{\max}^2$. Given a starting matrix $\Sigma_0$, the projected gradient descent scheme to minimize the barycenter functional of a measure $P$ supported on $\mathcal{K}_{\lambda_{\min},\lambda_{\max}}$ is given by
\begin{equation}\label{eqn:EGD_iteration}
    \Sigma_{n+1}^{\rm EGD}
     \deq  \Pi_{\lambda_{\min},\lambda_{\max}}\bigl(\Sigma_n - \eta\operatorname{D}F(\Sigma_n^{\rm EGD})\bigr)\,, \qquad n \geq 0\,.
\end{equation}
Also, suppose that $\Sigma_1,\dotsc,\Sigma_n$ are i.i.d.\ samples from $P$.
Then, the projected stochastic gradient scheme is
\begin{align}\label{eq:esgd_iteration}
    \Sigma_{n+1}^{\rm ESGD}
    & \deq  \Pi_{\lambda_{\min},\lambda_{\max}}\bigl(\Sigma_n^{\rm ESGD} - \eta_{n+1} \, \bigl\{I_d - \GM(\Sigma_{n+1}, {(\Sigma_n^{\rm ESGD})}^{-1})\bigr\}\bigr)\,, \qquad n \geq 0\,,
\end{align}
where following~\cite{lacostejulienschmidtbach2012projectedsgd} we take the step size to be $\eta_n = 8\lambda_{\max}^4/(\lambda_{\min}^3 \, (n+1))$.

We now state the convergence guarantees for these two algorithms.

\begin{thm}[Guarantees for Euclidean GD/SGD]\label{thm:egd}
Assume that $P$ is supported on covariance matrices whose eigenvalues lie in the range $[\lambda_{\min}, \lambda_{\max}]$, $0 < \lambda_{\min} \leq \lambda_{\max} < \infty$. Let $\kappa  \deq  \lambda_{\max}/\lambda_{\min}$ denote the condition number. Assume that we initialize at $\Sigma_0 \in \supp P$.
\begin{enumerate}
    \item (EGD) Let $\Sigma_n^{\rm EGD}$ denote the $n$-th iterate of projected Euclidean gradient descent~\eqref{eqn:EGD_iteration}.
    Then,
\begin{equation}
    \norm{\Sigma_n^{\rm EGD} - \Sigma^\star}_{\rm F}^2 \leq \exp\bigl(-\frac{n}{\kappa^6}\bigr) \, \norm{\Sigma_0 - \Sigma^\star}_{\rm F}^2\,.
\end{equation}
    \item (ESGD) Let $\Sigma_n^{\rm ESGD}$ denote the $n$-th iterate of Euclidean projected stochastic gradient descent~\eqref{eq:esgd_iteration}. Then,
    \begin{align*}
        \E[\norm{\Sigma_n^{\rm ESGD} - \Sigstar}_{\rm F}^2]
        &\le \frac{64d\lambda_{\max}^2 \kappa^{6.5}}{n}\,.
    \end{align*}
\end{enumerate}
\end{thm}
\begin{proof}
    (1) The preceding lemma shows that the barycenter functional $F$ is strongly convex and smooth with condition number $\kappa^6$.
By \cite[Theorem 3.10]{bubeck2015convex}, projected gradient descent \eqref{eqn:EGD_iteration} converges at the stated rate. 

(2) For ESGD, we must compute a bound on the Euclidean variance of the stochastic gradient. Using Lemma~\ref{lem:transport_map_reg}, we get the two-sided control
\begin{align*}
    \frac{1}{\sqrt{\kappa}}\, I_d \preceq  \Sigma_{n+1} \# {(\Sigma_n^{\rm ESGD})}^{-1} \preceq  \sqrt{\kappa}\, I_d
\end{align*}
and thus
\begin{align*}
    \bigl\lVert I_d - \Sigma_{n+1} \# {(\Sigma_n^{\rm ESGD})}^{-1} \bigr\rVert_{\rm F}^2
    &\le d\, (\sqrt{\kappa} - 1)
    \le d\sqrt{\kappa}\,.
\end{align*}
The result now follows from the preceding lemma and~\cite{lacostejulienschmidtbach2012projectedsgd}.
\end{proof}

\begin{rema}\label{rema:w2_vs_frob}
    To compare the guarantees of Theorems~\ref{thm:bures_gd} and~\ref{thm:egd}, first we have
    \begin{align*}
        \frac{1}{2} \, \norm{\Sigma_n^{1/2} - \Sigma^{\star \, 1/2}}_{\rm F}^2
        &\le W_2^2(\Sigma_n, \Sigstar)
        \le \norm{\Sigma_n^{1/2} - \Sigma^{\star \, 1/2}}_{\rm F}^2
    \end{align*}
    as a consequence of~\cite[Lemma 3.5]{carrillovaes2019covfokkerplanck}. Moreover, under our assumptions,
    \begin{align*}
        \frac{1}{4\lambda_{\max}} \, \norm{\Sigma_n - \Sigstar}_{\rm F}^2
        &\le \norm{\Sigma_n^{1/2} - \Sigma^{\star \, 1/2}}_{\rm F}^2
        \le \frac{1}{4\lambda_{\min}} \, \norm{\Sigma_n - \Sigstar}_{\rm F}^2\,,
    \end{align*}
    where the first inequality is elementary and follows from
    \begin{align*}
        A-B
        &= A^{1/2} \,(A^{1/2} - B^{1/2}) + (A^{1/2} - B^{1/2}) \, B^{1/2}\,,
    \end{align*}
    whereas the second inequality uses~\cite[(X.46)]{bhatia1997matrixanalysis}.
\end{rema}

For the iterations given by \eqref{eqn:EGD_iteration} and~\eqref{eq:esgd_iteration} to be practical, we need the projection step to be implementable. The following lemma takes care of this.

\begin{lemma}
Let $\Pi_{\alpha,\beta}: \sym \to \mathcal{K}_{\alpha,\beta}$ be the projection with respect to the Frobenius norm. Then
\begin{equation*}
    \Pi_{\alpha,\beta}(Y) = \sum\limits_{i=1}^d [(\lambda_i \wedge \beta) \vee \alpha] \,v_i v_i^\T 
\end{equation*}
where $Y = \sum_{i=1}^d \lambda_i v_i v_i^\T$ is an orthogonal eigendecomposition of $Y$. 
\end{lemma}
\begin{proof}
Let $Y=Q\Lambda Q^\T$ be an orthogonal eigendecomposition of $Y$. Since the Frobenius norm is unitarily invariant, we have
\begin{align*}
    \Pi_{\alpha,\beta}(Y) &= \argmin\limits_{X \in \mathcal{K}_{\alpha,\beta}}{\norm{X - Q\Lambda Q^\T}_\text{F}^2}
    = \argmin\limits_{X \in \mathcal{K}_{\alpha,\beta}}{\norm{Q^\T X Q - \Lambda}_\text{F}^2} 
    = Q\, \bigl(\argmin\limits_{X \in \mathcal{K}_{\alpha,\beta}}{\norm{X - \Lambda}_\text{F}^2} \bigr) \,Q^\T
\end{align*}
and the result follows. 
\end{proof}

Finally, we state and prove the elementary lemma we used in the proof of Lemma~\ref{lem:strong convexity smoothness}.

\begin{lemma}\label{lem:trace inequality}
Let $A,B \in \psd$ and $Y \in \sym$. Then
\begin{equation*}
    \lambda_{\min}(A) \, \lambda_{\min}(B) \, \norm{Y}_{\rm F}^2 \leq \tr(AYBY) \leq \lambda_{\max}(A) \, \lambda_{\max}(B) \, \norm{Y}_{\rm F}^2\,.
\end{equation*}
\end{lemma}
\begin{proof}
    The result follows immediately from $\tr(AYBY) = \norm{A^{1/2}YB^{1/2}}_\text{F}^2$ and $\lambda_{\min}(A^{1/2})= {\lambda_{\min}(A)}^{1/2}$ (similarly for $B$). 
\end{proof}

\subsection{SDP formulation}\label{scn:sdp}

The SDP formulation of the Bures--Wasserstein barycenter is as follows. Suppose that $P$ is a discrete distribution, $P = \sum_{i=1}^k p_i \delta_{\Sigma_i}$.
The Wasserstein distance between $\Sigma_0, \Sigma_1 \in \psd$ can be expressed as
\begin{align*}
    W_2^2(\Sigma_0,\Sigma_1)
    &= \min_{S\in \R^{d\times d}}\Biggl\{ \tr(\Sigma_0 + \Sigma_1 - 2S) \quad\text{such that}\quad \begin{bmatrix} \Sigma_0 & S \\ S^\T & \Sigma_1 \end{bmatrix} \succeq 0 \Biggr\}\,.
\end{align*}
It follows that the barycenter $\Sigstar$ of $P$ solves the optimization problem
\begin{align*}
    \operatorname*{minimize}_{\substack{\Sigstar \in \psd \\ S_1,\dotsc,S_k\in\R^{d\times d}}}\Biggl\{ \tr\Bigl(\Sigstar - 2\sum_{i=1}^k p_i S_i\Bigr) \quad\text{such that}\quad \begin{bmatrix} \Sigma_i & S_i \\ S_i^\T & \Sigstar \end{bmatrix} \succeq 0\,, \; \forall i \in [k]\Biggr\}\,.
\end{align*}

\section{Proofs for entropically-regularized barycenters}\label{sec:ent}

We begin by remarking how the non-centered case can be reduced to the centered case.

\begin{rema}\label{rmk:noncentered_entropically_reg}
    For a probability measure $\mu$, let $m_\mu$ denote its mean and let $\bar\mu$ denote the centered version of $\mu$. Using Fact~\ref{fact:non_zero_mean} in Appendix~\ref{scn:facts}, one can verify that
    \begin{align*}
        &\frac{1}{2} \int W_2^2(b, \mu) \, \D P(\mu) + \gamma \on{KL}\bigl(b \bigm\Vert \mc N(0, I_d)\bigr) \\
        &\qquad = \frac{1}{2} \int \norm{m_b - m_\mu}^2 \, \D P(\mu) + \frac{\gamma}{2} \, \norm{m_b}^2 + \frac{1}{2} \int W_2^2(\bar b, \bar \mu) \, \D P(\mu) + \gamma \on{KL}\bigl(\bar b \bigm\Vert \mc N(0, I_d)\bigr)\,.
    \end{align*}
    This shows that the objective of the entropically-regularized barycenter decouples into two parts, one involving the mean of $b$ and the other involving the centered version of $b$.
    Explicitly, we can compute
    \begin{align*}
        m^\star
        & \deq  \frac{1}{1+\gamma} \int m_\mu \, \D P(\mu)
    \end{align*}
    and the entropically-regularized barycenter $\bar b^\star$ of the centered versions of the distributions in $P$. Then, if $\tau : \R^d\to\R^d$ denotes the translation $x \mapsto x + m^\star$, the solution to the original entropically-regularized barycenter problem is $\tau_\# \bar b^\star$.
\end{rema}

We now overview the proof strategy; proofs are then provided in the subsequent subsections. Throughout this section let $P$ be supported
on $\mathcal{K}_{1/\sqrt{\kappa}, \sqrt{\kappa}}$, the subset of matrices in $\psd$ with eigenvalues in the range $[1/\sqrt\kappa,\sqrt\kappa]$.

An important observation driving our analysis
is that the gradient of the KL divergence at $\Sigma$ has the following form:
\begin{equation}\label{eqn:KL_grad}
\nabla \KL{\cdot}{I_d}(\Sigma) = I_d
- \Sigma^{-1} = I_d - T_{\Sigma \to \Sigma^{-1}} =
-\log_{\Sigma}(\Sigma^{-1})\,.
\end{equation}
This can be shown by observing that
\[\KL{\Sigma}{I_d}
    = \frac{1}{2} \tr \Sigma - \frac{1}{2} \ln \det \Sigma - \frac{d}{2}\,,
\]
computing the Euclidean gradient, and appealing to Fact~\ref{fact:bures_gradient} in Appendix~\ref{scn:facts}.
This gradient identity is convenient for applying our eigenvalue control and allows us to prove the following Lemma in Subsection~\ref{subsec:ent_trap}.
Put $\Sigma^+ \deq  \exp_{\Sigma}(-\eta \nabla F_{\gamma}(\Sigma))$.

\begin{lemma}\label{lem:ent_trap}
    Let $\lambda \deq {(2+\gamma)}^2 \,\sqrt\kappa$ and suppose that the step size satisfies $\eta \le \frac{2}{\lambda^2}$.
    If $\Sigma \in \mathcal{K}_{1/\lambda, \lambda}$, then so is $\Sigma^+$.
\end{lemma}
Throughout this section,
we thus use the notation $\lambda := (2 + \gamma)^2 \sqrt{\kappa}$. We also establish a couple of properties of our objective function in Subsection~\ref{subsec:ent_KL_facts}.

\begin{prop}\label{prop:ent_KL_facts}
Define $G : \mathcal{K}_{1/\lambda, \lambda} \to \R$ to take $\Sigma \mapsto \KL{\Sigma}{I_d}$.
Then, the following hold:
\begin{enumerate}
    \item $G$ is $2\lambda$-smooth with respect to Wasserstein geodesics.
    \item $F_\gamma$ is $1/(4\lambda^{7})$-strongly convex with respect to Euclidean geodesics on $\mc K_{1/\lambda,\lambda}$.
    \item $F_\gamma$ is strictly convex on all of $\psd$.
\end{enumerate}
\end{prop}

With these facts, we can establish existence and uniqueness of $\Sigma^\star$ and prove Proposition~\ref{prop:ent_unique} in
Subsection~\ref{subsec:ent_unique}.

Next we prove smoothness and PL inequalities in Subsection~\ref{subsec:ent_PL_smoothness}.

\begin{lemma}[Smoothness]\label{lem:ent_smoothness}
    If $\Sigma \in \mathcal{K}_{1/\lambda, \lambda}$ and we take the step size at most $\eta \le \frac{2}{\lambda^2}$,
then
$$
F_{\gamma}(\Sigma^+) - F_{\gamma}(\Sigma) \leq - \frac{\eta}{2}\,   \|\nabla F_{\gamma}(\Sigma)\|_{\Sigma}^2\,.
$$
\end{lemma}

\begin{lemma}[PL inequality] \label{lem:ent_PL}
    If $\Sigma \in \mathcal{K}_{1/\lambda, \lambda}$,
    then
    $$
    F_{\gamma}(\Sigma) - F_{\gamma}(\Sigma^\star) \leq  \frac{\lambda^8}{2} 
    \,\| \nabla F_{\gamma}(\Sigma) \|_{\Sigma}^2\,.
    $$
\end{lemma}

The main theorem now follows by combining these lemmas.

\begin{proof}[Proof of Theorem~\ref{thm:ent_main}]
    By Lemma~\ref{lem:ent_trap}, the Lemmas~\ref{lem:ent_smoothness} and~\ref{lem:ent_PL} hold throughout
    the optimization trajectory.
    Then,
    \begin{align*}
        F_{\gamma}(\Sigma_{t +1}) - F_{\gamma}(\Sigma^\star) &=
        F_{\gamma}(\Sigma_{t + 1}) - F_{\gamma}(\Sigma_t) + F_{\gamma}(\Sigma_t) -  F_{\gamma}(\Sigma^\star) \\
        &\leq - \frac{2}{\lambda^2}\, \| \nabla F_{\gamma}(\Sigma_t)\|_{\Sigma_t}^2
        +  F_{\gamma}(\Sigma_t) -  F_{\gamma}(\Sigma^\star) \\
        &\leq \Bigl( 1 - \frac{4}{\lambda^{10}} \Bigr)
        \, \{F_{\gamma}(\Sigma_t) - F_{\gamma}(\Sigma^\star)\}\,.
    \end{align*}
    Iterating yields the result.
\end{proof}

\subsection{Trapping the iterates}\label{subsec:ent_trap}

\begin{proof}[Proof of Lemma~\ref{lem:ent_trap}]
Combining~\eqref{eqn:KL_grad} with the formula for the gradient of the squared Bures--Wasserstein distance (Fact~\ref{fact:grad_sq_dist} in Appendix~\ref{scn:facts}),
we see that in fact
$$
-\nabla F_{\gamma}(\Sigma) = \int \log_{\Sigma}(\Sigma') \, \D P(\Sigma') + \gamma
\log_{\Sigma}(\Sigma^{-1})\,.
$$
Then, $\Sigma^+$ is the generalized barycenter of
\begin{align*}
    P_{\gamma,\eta}
    = (1-\eta-\gamma\eta) \, \delta_\Sigma + \eta \, P + \gamma\eta \, \delta_{\Sigma^{-1}}
    = (1-\eta-\gamma\eta) \, \delta_\Sigma + (\eta + \gamma\eta) \, \tilde P_{\gamma,\eta}\,,
\end{align*}
where
\begin{align*}
    \tilde P_{\gamma,\eta}
    &= \frac{\eta}{\eta +\gamma\eta} \, P + \frac{\gamma\eta}{\eta+\gamma\eta} \, \delta_{\Sigma^{-1}}\,.
\end{align*}
Let us search for $\lambda \ge 0$ such that if $\Sigma$ has eigenvalues in the range $[1/\lambda,\lambda]$, then so does $\Sigma^+$.
Suppose that $\lambda \ge \sqrt{\kappa}$.
First of all, by the generalized
geodesic convexity of $\lambda_{\max}$, we have $\lambda_{\max}(\Sigma^+) \le (1-\eta) \, \lambda + \eta \,\sqrt\kappa \le \lambda$.
Next, we calculate
\begin{align*}
    \beta
    &\deq \lambda \ge
   \frac{\eta}{\eta + \gamma\eta} \,\kappa^{1/2} + \frac{\gamma\eta}{\eta+\gamma\eta} \,\lambda\geqslant  
    \int \lambda_{\max} \, \D \tilde P_{\gamma,\eta}, \\
    &\Bigl(\int \sqrt{\lambda_{\min}} \, \D \tilde P_{\gamma,\eta}\Bigr)^2
    \ge \Bigl( \frac{\eta}{\eta +\gamma\eta} \, \kappa^{-1/4} + \frac{\gamma\eta}{\eta + \gamma \eta} \, \lambda^{-1/2}\Bigr)^2\,
    =: \alpha \, .
\end{align*}
Suppose, that in addition
to $\lambda \ge \sqrt{\kappa}$,
we can solve the equation
$\lambda^{-1} = 
\alpha /4$.
Then if we take
$\eta \le \frac{\alpha}{2\beta}
= \frac{2}{\lambda^2}$
we can
apply Theorem~\ref{thm:sqrt_lambda_min_concave} (see also the discussion in Appendix~\ref{scn:gen_geod_cvxty}) to deduce that $\lambda_{\min}(\Sigma^+) \ge \alpha/4 = \lambda^{-1}$.
Therefore, it suffices to choose $\lambda$ so that
both
\begin{align*}
\lambda \ge \sqrt{\kappa}\, ,
\quad \textrm{ and }
\quad 
   \lambda^{-1} = \frac{1}{4}\, \Bigl( \frac{\eta}{\eta +\gamma\eta} \, \kappa^{-1/4} + \frac{\gamma\eta}{\eta + \gamma \eta} \, \lambda^{-1/2}\Bigr)^2\,
.
\end{align*}
It can be verified that these
equations are solved by taking
$\lambda = (2 + \gamma)^2\, \sqrt{\kappa}$.
\end{proof}

\subsection{Properties of the KL divergence}\label{subsec:ent_KL_facts}

\begin{proof}[Proof of Proposition~\ref{prop:ent_KL_facts}]
For the first claim, fix $\Sigma_0, \Sigma_1 \in \mathcal{K}_{1/\lambda, \lambda}$
and let $T$ denote the transport map from $\Sigma_0$ to $\Sigma_1$.
Then put
$$
\Sigma_s  \deq  \bigl((1-s) I_d + s T\bigr) \Sigma_0 \bigl((1-s) I_d + s T\bigr)
= {(1 - s)}^2 \, \Sigma_0 + s^2\, \Sigma_1 + s\,(1 - s)\,(\Sigma_0 T
+ T \Sigma_0)\,.
$$ In other words, ${(\Sigma_s)}_{s \in [0,1]}$
is the Bures--Wasserstein geodesic between $\Sigma_0$
and $\Sigma_1$ (see~\eqref{eq:bures_geod}). It suffices to show (see Appendix~\ref{scn:geod_opt}) that
$$
\partial_s^2 \KL{\Sigma_s}{I_d} |_{s = 0} \leq 2\lambda \,W_2^2(\Sigma_0, \Sigma_1).
$$ Since $\KL{\Sigma_s}{I_d} = \frac{1}{2}(\tr(\Sigma_s) - \ln \det \Sigma_s + \text{constant})$,
we analyze the first two terms separately.
First, we note that 
$$
\partial_s^2 \tr(\Sigma_s)|_{s = 0} = 2\tr(\Sigma_0 + \Sigma_1 - 2 \Sigma_0
T) = 2W_2^2(\Sigma_0, \Sigma_1)\,,
$$
where the equality follows from~\eqref{eq:alternate_w2_gaussian}.
For the second term we start by observing
that
$$
-\ln \det \Sigma_s = -\ln \det \Sigma_0 - 2\ln \det\bigl((1 - s)I_d
+ s T\bigr)\,.
$$
Using this identity we see that
$$
-\partial_s^2\bigl( \frac{1}{2} \ln \det (\Sigma_s)\bigr)\big\vert_{s = 0}
= \tr \bigl({(T - I_d)}^2\bigr) \leq \lambda \, \norm{T - I_d}_{\Sigma_0}^2
= \lambda \,W_2^2(\Sigma_0, \Sigma_1)\,.
$$ Putting these bounds together yields the result.

For the second claim, using the convexity of $-\ln \det$, it follows that the Euclidean Hessian of $F_\gamma$ satisfies $\on D^2 F_\gamma \succeq \on D^2 F$, where $F$ is the barycenter functional.
It follows from Lemma~\ref{lem:strong convexity smoothness} that
\begin{align*}
    \on D^2 F_\gamma \succeq \frac{1}{4\lambda^{7}} \, I_d
\end{align*}
on the set $\mc K_{1/\lambda,\lambda}$.

Finally, for the third claim we can use the convexity of $F$ and the strict convexity of $-\ln \det$ (together with $\gamma > 0$) to argue that $\on D^2 F_\gamma \succ 0$ on $\psd$.
\end{proof}

\subsection{Existence and uniqueness of the minimizer}\label{subsec:ent_unique}

\begin{proof}[Proof of Proposition~\ref{prop:ent_unique}]
First, we prove that when restricted to $\psd$, the functional $F_\gamma$ has a unique minimizer.
Let $\lambda = (2 + \gamma)^2\sqrt{\kappa}$
and $\eta \le \frac{4}{\lambda^2}$ be as in Lemma~\ref{lem:ent_trap} and let $H \colon \mathcal{K}_{1/\lambda, \lambda} \to \mathbb{S}_{++}^d$
take
$$
\Sigma \mapsto \exp_{\Sigma}\bigl(-\eta \,\nabla F_{\gamma}(\Sigma)\bigr)\,.
$$ Then by Lemma~\ref{lem:ent_trap},
$H$ maps $\mathcal{K}_{1/\lambda, \lambda}$ to itself.
We may thus apply
Brouwer's fixed point theorem to guarantee a fixed point of $H$ in $\mathcal{K}_{1/\lambda,
\lambda}$, call it $\Sigma^\star$. Note that this means precisely that $\nabla F_\gamma(\Sigma^\star) = 0$. By the equivalence of Euclidean and Bures--Wasserstein gradients (Fact~\ref{fact:bures_gradient} in Appendix~\ref{scn:facts}), we conclude that $\on D F_\gamma(\Sigma^\star) = 0$ as well. By the strict convexity of $F_\gamma$ (the third claim of Proposition~\ref{prop:ent_KL_facts}), we deduce that $\Sigma^\star$ is the unique minimizer of $F_\gamma$ on $\psd$ (actually, on all of $\mbb S_+^d$, since $-\ln \det$ blows up if the determinant approaches $0$).

Next, let $b$ be a probability measure on $\R^d$ which has mean $m$ and covariance matrix $\Sigma$. Let $\bar b$ denote the centered version of $b$. We now claim that
\begin{align*}
    F_\gamma(b)
    &\ge F_\gamma(\bar b)
    \ge F_\gamma\bigl(\mc N(0, \Sigma)\bigr)
    \ge F_\gamma\bigl(\mc N(0, \Sigma^\star)\bigr)\,.
\end{align*}
The first inequality is due to Remark~\ref{rmk:noncentered_entropically_reg} and it is strict unless $b = \bar b$. The second inequality follows from Fact~\ref{fact:low_bdd_gaussian} in Appendix~\ref{scn:facts}, together with the classical fact that the Gaussian maximizes entropy among all centered distributions with the same covariance matrix; this latter fact is proven in~\cite[Theorem 8.6.5]{coverthomas2006infotheory}, and it also shows that the inequality is strict unless $\bar b = \mc N(0, \Sigma)$. Finally, the last inequality is what we have shown above, and it is also strict unless $\Sigma=\Sigma^\star$.
\end{proof}

\subsection{Smoothness and PL inequalities}\label{subsec:ent_PL_smoothness}

\begin{proof}[Proof of Lemma~\ref{lem:ent_smoothness}]
    By Lemma~\ref{lem:ent_trap}, Proposition~\ref{prop:ent_KL_facts}, and the $1$-geodesic smoothness of the barycenter functional~\cite[Theorem 7]{chewietal2020buresgd} we deduce that $F_\gamma = F + \gamma G$ is $(1 + 2\gamma \lambda)$-smooth, i.e.,
    \begin{align*}
    F_{\gamma}(\Sigma^+) - F_{\gamma}(\Sigma) &\leq \langle \nabla F_{\gamma}(\Sigma), \log_{\Sigma}(
    \Sigma^+) \rangle_{\Sigma} + \frac{1 + 2\gamma\lambda}{2}\, W_2^2(\Sigma, \Sigma^+) \,.
    \end{align*}
    Substituting in $\log_\Sigma(\Sigma^+) = -\eta \nabla F_\gamma(\Sigma)$ and using $\frac{2}{\lambda^2} \le \frac{1}{1+2\gamma\lambda}$ yields the result.
\end{proof}

\begin{proof}[Proof of Lemma~\ref{lem:ent_PL}]
    From the second claim in Proposition~\ref{prop:ent_KL_facts}, 
    and since $\mathcal{K}_{1/\lambda, \lambda}$ is convex with respect to Euclidean
    geodesics, we see that for $\Sigma \in \mathcal{K}_{1/\lambda, \lambda}$
    \begin{align*}
        F_{\gamma}(\Sigma) - F_{\gamma}(\Sigma^\star)
        &\leq \langle \on D F_{\gamma}(\Sigma), \Sigma - \Sigma^\star \rangle
    - \frac{1}{8\lambda^{7}} \,\|\Sigma - \Sigma^\star\|_{\rm F}^2 \\
    &= \frac{1}{2} \, \langle \nabla F_{\gamma}(\Sigma), \Sigma - \Sigma^\star \rangle
    - \frac{1}{8\lambda^{7}} \,\|\Sigma - \Sigma^\star\|_{\rm F}^2\,,
    \end{align*}
    where the last line uses Fact~\ref{fact:bures_gradient} in Appendix~\ref{scn:facts}.
    Next we observe that by combining Cauchy--Schwarz with Young's inequality we get
    that for all $r > 0$,
    \begin{align*}
    \frac{1}{2} \, \langle \nabla F_{\gamma}(\Sigma), \Sigma - \Sigma^\star \rangle
    &\le \frac{1}{2} \, \norm{\nabla F_{\gamma}(\Sigma)}_{\Sigma}\, \norm{\Sigma - \Sigma^\star}_{\Sigma^{-1}}
    \le \frac{r}{16} \,\|\nabla F_{\gamma}(\Sigma)\|_{\Sigma}^2
    + \frac{1}{r} \, \|\Sigma - \Sigma^\star\|_{\Sigma^{-1}}^2 \\
    &\leq  \frac{r}{16} \,  \|\nabla F_{\gamma}(\Sigma)\|_{\Sigma}^2
    + \frac{\lambda}{r} \, \|\Sigma  - \Sigma^\star\|_{\rm F}^2\,.
    \end{align*} Putting $r = 8\lambda^8$ yields the result.
\end{proof}

\section{Proofs for geometric medians}\label{app:median}

\subsection{Convergence guarantee for smoothed Riemannian gradient descent}\label{scn:proofs_median}

We begin with the proof of Proposition~\ref{prop:basic_prop_median}.

\begin{proof}[Proof of Proposition~\ref{prop:basic_prop_median}]
    Let $F : \mc P_2(\R^d) \to \R$ be the geometric median functional, defined via $F(b)  \deq  \int W_2(b, \cdot) \, \D P$. If we regard $F$ as a functional over the Bures--Wasserstein space, then by continuity of $F$ and compactness of the set $\{\norm\cdot \le \lambda_{\max}\} \subseteq \mbb S_+^d$, there exists a minimizer $\Sigstar_{\rm median}$ of $F$ on this set. We will show that the Gaussian $b^\star_{\rm median}$ with covariance $\Sigstar_{\rm median}$ minimizes $F$ over all of Wasserstein space.
    
    First, recall the map $\clip^{\lambda_{\max}}$ in Proposition~\ref{thm:eig_clip_bures_contraction}, which is a contraction w.r.t.\ the Bures--Wasserstein metric.
    Then, for any $\Sigma \in \mbb S_+^d$, it holds that
    \begin{align*}
        F(\Sigma)
        = \int W_2(\Sigma, \Sigma') \, \D P(\Sigma')
        &\ge \int W_2(\clip^{\lambda_{\max}}\Sigma, \Sigma') \, \D P(\Sigma') \\
        &\ge \int W_2(\Sigstar_{\rm median}, \Sigma') \, \D P(\Sigma')
        = F(\Sigstar_{\rm median})\,,
    \end{align*}
    so that $\Sigstar_{\rm median}$ minimizes $F$ over $\mbb S_+^d$.
    
    Next, using Fact~\ref{fact:low_bdd_gaussian} in Appendix~\ref{scn:facts}, if $b \in \mc P_2(\R^d)$ has covariance matrix $\Sigma$, then
    \begin{align*}
        F(b)
        &= \int W_2(b, \cdot) \, \D P
        \ge \int W_2(\gamma_{0,\Sigma}, \cdot) \, \D P
        \ge \int W_2(\Sigstar_{\rm median}, \cdot) \, \D P
        = F(b^\star_{\rm median})\,,
    \end{align*}
    so that $b^\star_{\rm median}$ minimizes $F$ over $\mc P_2(\R^d)$.
    By definition, $\Sigstar_{\rm median}$ has eigenvalues upper bounded by $\lambda_{\max}$, which completes the proof.
\end{proof}

As the main difficulty in the analysis of the geometric median is the lack of both convexity and smoothness, we now pause to justify these remarks.

\begin{rema}\label{rmk:w2_horrible}
    We claim that the unsquared Wasserstein distance $W_2(\cdot, \Sigma')$ is neither geodesically convex nor geodesically smooth. For the former statement, note that the geodesic convexity of $W_2(\cdot,\Sigma')$ would imply the geodesic convexity of $W_2^2(\cdot,\Sigma')$, but the squared Wasserstein distance is known to not be geodesically convex (in general, it is not even \emph{semi-convex}, see~\cite[Example 9.1.5]{ambrosio2008gradient}; for a Gaussian example, see~\cite[Appendix B.2]{chewietal2020buresgd}). In fact, unsquared metrics are almost never geodesically smooth; if this were the case, then there would exist a constant $\beta < \infty$ for which $W_2(\Sigma,\Sigma') \le \frac{\beta}{2} \, W_2^2(\Sigma,\Sigma')$, which is manifestly false.
    
    Moreover, the function $W_2(\cdot,\Sigma')$ is neither Euclidean convex nor Euclidean smooth. To see this, observe that in one dimension we have $W_2(\Sigma,\Sigma') = \abs{\sqrt{\Sigma}-\sqrt{\Sigma'}}$ (see Fact~\ref{fact:diagonal} in Appendix~\ref{scn:facts}), which is neither convex nor smooth. It is notable that for this one-dimensional example, $W_2(\Sigma,\Sigma')$ is convex with respect to the variable $\sqrt{\Sigma}$, but it appears that reparameterization does not help in general; numerics indicate that the function $A \mapsto W_2(A^2, \Sigma')$ is not Euclidean convex on $\psd$.
\end{rema}

We now proceed with the analysis of the smoothed Riemannian gradient descent algorithm given as Algorithm~\ref{ALG:smoothed}. Recall that $F_\varepsilon$ denotes the smoothed geometric median functional.
The first step is to show that the smoothing does not affect the objective significantly.

\begin{lemma}\label{lem:smoothing_the_fn}
    For any $\Sigma \in \psd$, we have $\abs{F(\Sigma) - F_\varepsilon(\Sigma)} \le \varepsilon$.
\end{lemma}
\begin{proof}
    This follows from
    \begin{align*}
        \abs{W_2(\Sigma,\Sigma') - \sqrt{W_2^2(\Sigma, \Sigma') + \varepsilon^2}}
        &= \abs{\sqrt{W_2^2(\Sigma,\Sigma')} - \sqrt{W_2^2(\Sigma, \Sigma') + \varepsilon^2}}
        \le \varepsilon
    \end{align*}
    and integrating.
\end{proof}

We next show that replacing $W_2$ by $W_{2,\varepsilon}$ indeed yields smoothness.

\begin{lemma}\label{lem:med_geod_smooth}
    The functional $F_\varepsilon$ is $1/\varepsilon$-geodesically smooth.
\end{lemma}
\begin{proof}
Recall from Theorem~\ref{thm:properties_of_bary_fn} that one-half of the squared Wasserstein distance is $1$-smooth.
This means that for any $W_2$ geodesic ${(\Sigma_t)}_{t\in\R}$, the following Hessian bound holds:
\begin{align*}
    \frac{1}{2} \,\partial_t^2|_{t=0} W_2^2(\Sigma_t, \Sigma')
    &\le \norm{\dot \Sigma_0}^2_{\Sigma_0}\,.
\end{align*}
Here, $\dot\Sigma_0$ denotes the Bures--Wasserstein tangent vector, see the end of Appendix~\ref{scn:bures_geometry}.
We use this to compute the smoothness of $F_\varepsilon$.
Riemannian calculus yields
\begin{align*}
    \partial_t F_\varepsilon(\Sigma_t)
    &= \int \frac{\partial_t W_2^2(\Sigma_t, \Sigma')}{2W_{2,\varepsilon}(\Sigma_t, \Sigma')} \, \D P(\Sigma')\,, \\
    \partial_t^2|_{t=0} F_\varepsilon(\Sigma_t)
    &= \int \Bigl[ \frac{\partial_t^2|_{t=0} W_2^2(\Sigma_t, \Sigma')}{2W_{2,\varepsilon}(\Sigma_0,\Sigma')} - \frac{{\{\partial_t|_{t=0} W_2^2(\Sigma_t, \Sigma')\}}^2}{4W_{2,\varepsilon}^3(\Sigma_0,\Sigma')} \Bigr] \, \D P(\Sigma')\,.
\end{align*}
The second term is non-positive.
For the first term,
\begin{align*}
    \int \frac{\partial_t^2|_{t=0} W_2^2(\Sigma_t, \Sigma')}{2W_{2,\varepsilon}(\Sigma_0,\Sigma')} \, \D P(\Sigma')
    &=\int \underbrace{\frac{1}{W_{2,\varepsilon}(\Sigma_0,\Sigma')}}_{\le 1/\varepsilon} \, \underbrace{\frac{1}{2} \, \partial_t^2|_{t=0} W_2^2(\Sigma_t, \Sigma')}_{\le \norm{\dot\Sigma_0}_{\Sigma_0}^2} \, \D P(\Sigma')
    \le \frac{1}{\varepsilon} \, \norm{\dot \Sigma_0}_{\Sigma_0}^2\,.
\end{align*}
Hence, $F_\varepsilon$ is $1/\varepsilon$-smooth.
\end{proof}

In order to proceed with the analysis, we must study the dynamics of the smoothed Riemannian gradient descent algorithm.
To study these dynamics, it is helpful to again adopt the notation and calculus of general Wasserstein space.
The Wasserstein gradient of $F_\varepsilon$ is
\begin{align*}
    \nabla F_\varepsilon(b)
    &= -\int \frac{T_{b\to\mu} - {\id}}{W_{2,\varepsilon}(b,\mu)} \, \D P(\mu)\,,
\end{align*}
and one step of the Wasserstein gradient descent iteration with step size $\eta$ is
\begin{align*}
    b^+
    & \deq  \exp\bigl(- \eta \nabla F_\varepsilon(b)\bigr)
    = \Bigl[ {\id} + \eta \int \frac{T_{b\to\mu} - \id}{W_{2,\varepsilon}(b,\mu)} \, \D P(\mu) \Bigr]_\# b\,.
\end{align*}
We will rewrite this in the following way.
Define the weight
\begin{align*}
    \rho(\mu)
     \deq  \frac{{W_{2,\varepsilon}(b,\mu)}^{-1}}{\int {W_{2,\varepsilon}(b,\cdot)}^{-1} \, \D P}\,.
\end{align*}
Then,
\begin{align*}
    &{\id} + \eta \int \frac{T_{b\to\mu} - \id}{W_{2,\varepsilon}(b,\mu)} \, \D P(\mu) \\
    &\quad = \Bigl(1 - \eta \int W_{2,\varepsilon}(b,\cdot)^{-1} \, \D P \Bigr) \, {\id} + \Bigl(\eta \int W_{2,\varepsilon}(b,\cdot)^{-1}\,\D P\Bigr) \int T_{b\to\mu} \, \rho(\mu) \, \D P(\mu)\,.
\end{align*}
Since $\int {W_{2,\varepsilon}(b,\cdot)}^{-1} \, \D P \le 1/\varepsilon$, this is a convex combination of two terms if $\eta \le \varepsilon$.
Let us call the weights $1-\lambda$ and $\lambda$ respectively.
If we define the probability measure $\tilde P  \deq  (1-\lambda) \delta_b + \lambda \rho P$, then this can also be written as
\begin{align*}
    b^+
    &= \Bigl( \int T_{b\to\mu} \, \D \tilde P(\mu) \Bigr)_\# b\,.
\end{align*}

This expression proves the following fact (see also Appendix~\ref{scn:gen_geod_cvxty}).

\begin{lemma}
    The next iterate $b^+$ of smoothed Riemannian gradient descent starting at $b$ (with step size $\eta$) is a generalized barycenter of the distribution $\tilde P$ with base $b$.
\end{lemma}

We can now prove Theorem~\ref{thm:median_guarantee}.

\begin{proof}[Proof of Theorem~\ref{thm:median_guarantee}]
    By smoothness (Lemma~\ref{lem:med_geod_smooth}), if we take step size $\eta = \varepsilon$,
    \begin{align*}
        F_\varepsilon(\Sigma_{t+1}) - F_\varepsilon(\Sigma_t)
        &\le - \frac{\varepsilon}{2} \, \norm{\nabla F_\varepsilon(\Sigma_t)}_{\Sigma_t}^2\,.
    \end{align*}
    The result follows by telescoping.
\end{proof}

\subsection{Reduction for non-zero means}\label{scn:non_centered_median}

In this section, we suppose that $P$ is supported on non-degenerate, not necessarily centered Gaussians, whose covariance matrices have eigenvalues bounded above by $\lambda_{\max}$. We begin with the observation that if $b^\star_{\rm median}$ denotes a Gaussian minimizer of the median functional for $P$, then the mean of $b^\star_{\rm median}$ is not necessarily the Euclidean geometric median of the means of distributions in $\supp P$. To see this, consider the case when the Gaussians are one-dimensional.
Then, if we identify each Gaussian $\mu \in \supp P$ with its mean and \emph{standard deviation} (the square root of the variance) $(m_\mu, \sigma_\mu)$, then the $W_2$ distance between Gaussians is isometric to the standard Euclidean metric on the pairs $(m,\sigma)$ in $\R^2$ (see Facts~\ref{fact:diagonal} and~\ref{fact:non_zero_mean} in Appendix~\ref{scn:facts}). Therefore, the Wasserstein geometric median of $P$ is equivalent to the Euclidean geometric median of the pairs $(m,\sigma)$, and the statement whose validity is being investigated is tantamount to asking: is the first coordinate of the Euclidean geometric median in $\R^2$ equal to the median of the first coordinates? This statement is manifestly false.

Next, we describe the reduction. Let $(m, \Sigma)$, $(m', \Sigma')$ denote two pairs of means and covariance matrices in the support of $P$. Then,
\begin{align*}
    W_2^2\bigl((m, \Sigma), (m', \Sigma')\bigr)
    &= \norm{m-m'}^2 + W_2^2(\Sigma,\Sigma')\,,
\end{align*}
where the LHS denotes the squared Wasserstein distance between Gaussians with parameters $(m,\Sigma)$ and $(m,\Sigma')$ respectively.
The idea behind the reduction is that since the Wasserstein metric on diagonal matrices is the same as the Euclidean metric between the \emph{square roots} of the matrices (Fact~\ref{fact:diagonal} in Appendix~\ref{scn:facts}), we can embed the mean vectors as diagonal matrices, and take the direct sum of these diagonal matrices with the covariance matrices to form augmented matrices; then, we can apply the geometric median algorithm (Algorithm~\ref{ALG:smoothed}) to the augmented matrices. In this reduction, however, we must take care that when we embed the mean vectors, we embed them into \emph{positive definite} diagonal matrices.

Hence, define the augmented matrices
\begin{align*}
    \mb\Sigma  \deq  \begin{bmatrix} \diag({(m + C)}^2) & \\ & \Sigma \end{bmatrix}\,, \qquad \mb\Sigma'  \deq  \begin{bmatrix} \diag({(m' + C)}^2) & \\ & \Sigma' \end{bmatrix}\,,
\end{align*}
where the constant $C \ge \max\{\norm m_\infty, \norm{m'}_\infty\}$ is chosen to ensure that $\mb \Sigma,\mb \Sigma' \succeq 0$ (and that $m + C, m' + C \ge 0$). The Wasserstein distance between the augmented matrices is
\begin{align*}
    W_2^2(\mb \Sigma,\mb \Sigma')
    &= \norm{(m+C) - (m'+C)}_{\rm F}^2 + W_2^2(\Sigma, \Sigma')
    = \norm{m-m'}^2 + W_2^2(\Sigma, \Sigma') \\
    &= W_2^2\bigl((m, \Sigma), (m', \Sigma')\bigr)\,.
\end{align*}
Hence, after preprocessing the mean vectors and covariance matrices to form these augmented matrices, we may apply Algorithm~\ref{ALG:smoothed} to the augmented matrices in a black box manner. It is easy to check that the set of such diagonal block matrices
(where the upper block is itself diagonal) is convex under generalized geodesics.
Hence, as long as the Algorithm~\ref{ALG:smoothed} is initialized
at such a matrix every iterate will remain in that form, and therefore the iterates
will, when transformed back through the augmentation operation described above,
indeed approach a stationary point for the original median problem.
The convergence guarantee of Theorem~\ref{thm:median_guarantee} then applies.

Of course, it is likely that analyzing smoothed Riemannian gradient descent directly for the non-centered case could produce sharper results,
but this simple approach already gives dimension-free convergence rates for the Bures--Wasserstein geometric median.

\section{Further experiments and details}\label{scn:experiment_details}

\paragraph*{Reproducibility details.} Input generation details for Figures~\ref{fig:gd and sgd vs d},~\ref{fig:high-precision barycenter},~\ref{fig:median precision}, and~\ref{fig:robustness} are provided in the main text. For Figures~\ref{fig:sgd vs esgd} and~\ref{fig:effect of regularization}, recall that we generated matrices from a distribution whose barycenter is known to be the identity. By \cite[Theorem 2]{zemel2019procrustes}, if the mean of the distribution $(\log_{I_d})_\# P$ is $0$, then $I_d$ is the barycenter of $P$. In particular, if $Q$ is a mean zero distribution supported on symmetric matrices that lie in the domain of the exponential map, then $P = (\exp_{I_d})_\# Q$ has $I_d$ as its barycenter. In our experiments, we defined $Q$ to be the law of a random matrix with Haar eigenbasis and uniform eigenvalues from the interval $[-(1-\delta), 1-\delta]$ for a parameter $\delta \in (0,1)$. At the identity, the exponential map takes the simple form $\exp_{I_d} S = (I_d + S)^2$ and we see that $P$ is then supported on covariance matrices with spectrum in $[\delta^2, (2-\delta)^2]$. Both figures were generated with $\delta=0.1$. All experiments were performed using Julia 1.5.1 on a desktop computer running Ubuntu 18.04 with an Intel i7-10700 CPU. 

\paragraph*{Further empirical comparisons.} Here we further investigate the comparison of Riemannian and Euclidean GD done in~\autoref{fig:high-precision barycenter} by demonstrating qualitatively similar results for a variety of synthetic datasets.
For each dataset, 
the measure $P$ is the empirical measure of $n$ matrices of dimension $d \times d$ that are drawn randomly as follows.
\begin{enumerate}
    \item Haar eigenbasis and linearly spaced eigenvalues in $[\alpha, \beta]$.
    \item\label{item:synthetic data} Haar eigenbasis and i.i.d. $\operatorname{Unif}[\alpha, \beta]$ eigenvalues.
    \item\label{item:synthetic data 3} First split the matrices into $3$ groups. Each matrix has Haar eigenbasis and i.i.d.\ $\operatorname{Unif}[\alpha, \beta]$ eigenvalues where $[\alpha,\beta
    ] = 10^{i}\times[1, \kappa]$ for $i \in \{-2,0,2\}$ depending on its group. 
    \item Same as method \ref{item:synthetic data} above, except all matrices have the same eigenbasis. (Note that GD converges in $1$ step here since the matrices commute.)
    \item\label{item:synthetic data 5} Haar eigenbasis and eigenvalues uniform on a set of size $m \leq d$, whose elements are i.i.d.\ $\operatorname{Unif}[\alpha, \beta]$. 
    \item Same as method \ref{item:synthetic data 5} above, except all matrices use the same eigenvalues.
    \item\label{item:synthetic mix} Mix of all methods above. 
\end{enumerate}

Figures~\ref{fig:bary synthetic kappa=1000} and~\ref{fig:bary synthetic kappa=2} compare Euclidean and Riemannian GD on the barycenter problem as in Figure~\ref{fig:high-precision barycenter}, but now with these $7$ different input families. We average well-conditioned matrices in~\autoref{fig:bary synthetic kappa=1000}, and ill-conditioned matrices in~\autoref{fig:bary synthetic kappa=2}. The plots are generated using $n=d=50$ and $m=d/4$. For Method \ref{item:synthetic mix}, the $50$ matrices are divided into $6$ groups of roughly equal size. The $y$-axis measures the $W_2^2$ distance to the best iterate; and the $x$-axis measures time in seconds.

In these figures we had to hand-tune the stepsize for Euclidean GD since the stepsize indicated by Theorem~\ref{thm:egd} performs quite poorly. We used the same range of stepsizes ($\eta \in \{15,25,40\}$) in all plots to demonstrate that the performance of Euclidean GD is quite sensitive to its stepsize. In contrast, GD performs well on all inputs with its (untuned) stepsize of $1$. 

\begin{figure}
    \centering
    \includegraphics[width=\textwidth]{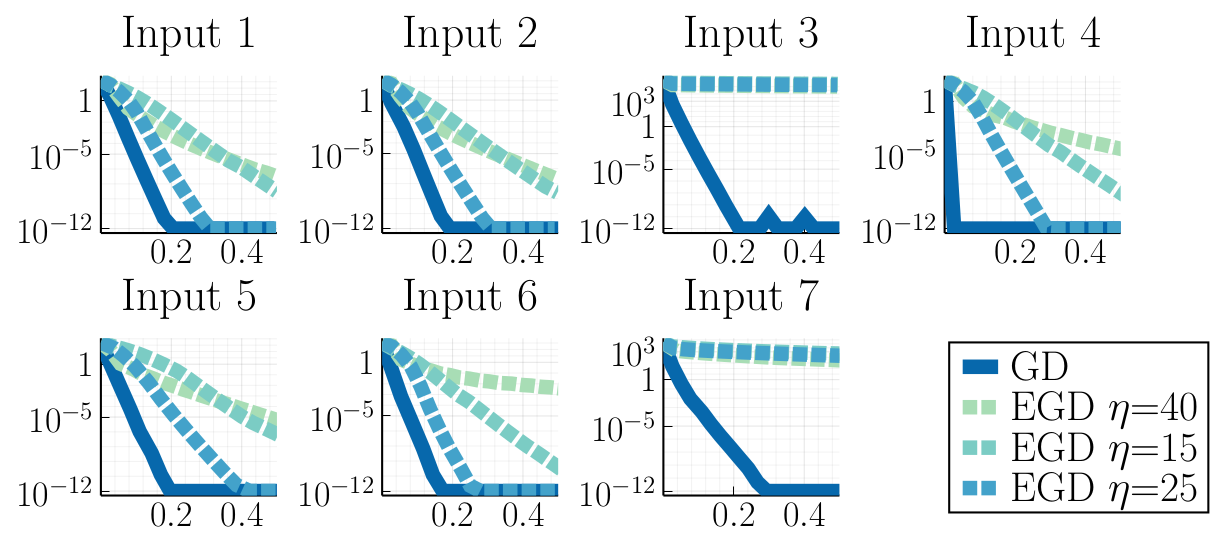}
    \caption{
        Comparison of high-precision barycenter algorithms for various types of synthetic data. Here, the matrices are poorly conditioned ($[\alpha, \beta] = [0.03,30]$ whereby $\kappa=1000$).
    }
    \label{fig:bary synthetic kappa=1000}
\end{figure}
\begin{figure}
    \centering
    \includegraphics[width=\textwidth]{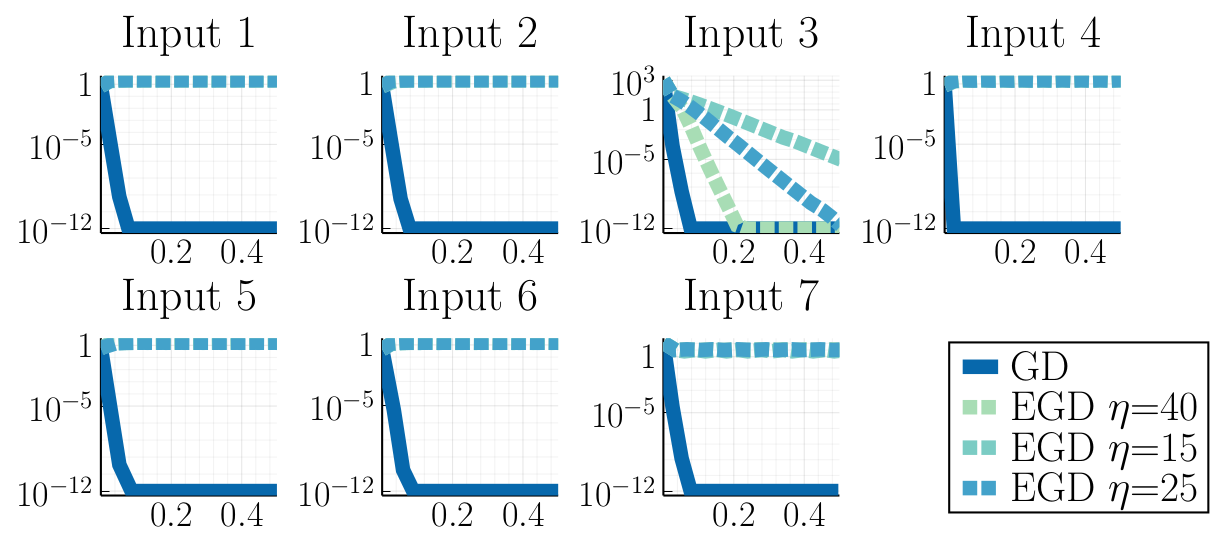}
    \caption{
        Comparison of high-precision barycenter algorithms for various types of synthetic data. Here, the matrices are well-conditioned ($[\alpha, \beta] = [1,2]$ whereby $\kappa = 2$).
    }
    \label{fig:bary synthetic kappa=2}
\end{figure}

\newpage
\RaggedRight{}
\printbibliography{}

\appendix

\end{document}